\documentclass{amsart}
\usepackage[utf8]{inputenc}
\usepackage[ruled, vlined]{algorithm2e}
\usepackage{amsmath}
\usepackage{amssymb}
\usepackage{amsthm}
\usepackage{amsrefs}
\usepackage{color}
\usepackage{graphicx}
\usepackage{epstopdf} 
\usepackage[margin=1in]{geometry}
\usepackage{caption}
\usepackage{booktabs}
\usepackage{verbatim}
\usepackage[shortlabels]{enumitem}
\usepackage[version=4]{mhchem}

\usepackage{subcaption}

\newtheorem{theorem}{Theorem}
\newtheorem{lemma}{Lemma}
\newtheorem{assumption}{Assumption}
\newtheorem{corollary}[lemma]{Corollary}
\newtheorem{prop}[lemma]{Proposition}

\DeclareMathOperator{\tr}{tr}
\DeclareMathOperator{\sgn}{sgn}
\DeclareMathOperator{\var}{Var}
\DeclareMathOperator{\spn}{span}

\newcommand{\bd}[1]{\boldsymbol{#1}}

\newcommand{\abs}[1]{\lvert#1\rvert}
\newcommand{\Abs}[1]{\left\lvert#1\right\rvert}

\newcommand{\norm}[1]{\lVert#1\rVert}
\newcommand{\Norm}[1]{\left\lVert#1\right\rVert}

\newcommand{\EE}{\mathbb{E}\,}

\begin{document}

\title[FCIQMC in the lens of inexact power iteration]{The full configuration interaction quantum Monte Carlo method in the lens of inexact power iteration}

\author{Jianfeng Lu} 
\address{Department of Mathematics, Department of
  Physics and Department of Chemistry, Duke University}
\email{jianfeng@math.duke.edu}

\author{Zhe Wang} \address{Department of
  Mathematics, Duke University}
\email{zhe.wang3@duke.edu}

\thanks{This research is supported in part by National Science
  Foundation under award DMS-1454939. We thank useful discussions with
  George Booth, Yingzhou Li, Jonathan Weare, Stephen Wright, and
  Lexing Ying during various stages of the work.}

\begin{abstract}
  In this paper, we propose a general analysis framework for inexact power iteration, which can be used to efficiently solve high dimensional eigenvalue problems arising from quantum many-body problems. Under the proposed framework, we establish the convergence theorems for several recently proposed randomized algorithms, including the full configuration interaction quantum Monte Carlo (FCIQMC) and the fast randomized iteration (FRI). The analysis is consistent with numerical experiments for physical systems such as Hubbard model and small chemical molecules. We also compare the algorithms both in convergence analysis and numerical results.
\end{abstract}

\maketitle

\section{Introduction}
In recent years, following the work of full configuration interaction
quantum Monte Carlo (FCIQMC) \cite{BoothThomAlavi:09,
  ClelandBoothAlavi:10}, the idea of using randomized or truncated
power method to solve the full configuration interaction (FCI)
eigenvalue problem has become quite popular in quantum chemistry
literature. From a mathematical point of view, the FCI calculation
essentially asks for the smallest eigenvalue of a real symmetric
matrix (for ground state calculation) or a few low-lying eigenvalues
(for low-lying excited state calculation). The computational challenge
lies in the fact that the size of the matrix grows exponentially fast
with respect to the number of orbitals / electrons in the chemical
system, and thus a brute-force numerical diagonalization method (such
as power method or Lanczos method) does not work except for very small
molecules.

The goal of this work is two folds: On the one hand, we want to
establish a general framework to understand these recently proposed
randomized algorithms. As we shall see, from the angle of numerical
linear algebra, these recent methods can be understood as
generalizations of conventional power method when inexact
matrix-vector product is used. As a result, the convergence of these
methods can be dealt with by a simple extension of the usual proof of
convergence of power method. A natural consequence of this
understanding is that, to compare the various approaches, the crucial
part is to understand the error caused by different strategies of
inexact matrix-vector multiplication. Using this insights, we will
compare a few of the recently proposed randomized or truncated FCI
methods analytically and also numerically using Hubbard model and some
small chemical molecules as toy examples.

While the motivation of the study is from FCI calculation in quantum
chemistry, these methods can be understood on the general setting of
numerical linear algebra, and hence except in the numerical section,
we will not restrict ourselves to the FCI Hamiltonian.  For a given
real symmetric positive definite matrix $A\in \mathbb{R}^{N\times N}$,
we are interested in numerically obtaining the largest eigenvalue and
corresponding eigenvector. It is possible to extend the method to
leading $k$ eigenvalues where $k$ is on the order of $1$ based on the
subspace iteration method, generalization of the power method.  In the
sequel, we denote the eigenvalues of $A$ as
$\lambda_1 \geq \lambda_2 \geq \cdots \geq \lambda_N \geq 0$, and
corresponding orthonormal eigenvectors are $u_1, u_2, \cdots, u_N$
(viewed as column vectors).

To obtain the largest eigenvalue and the corresponding eigenvector,
one of the simplest algorithm is standard power iteration, given by
\begin{align*}
  & y_{t+1} = Ax_{t}; \\
  & x_{t+1} = y_{t+1}/ \norm{y_{t+1}}_2 
\end{align*}
with some initial guess $x_0$ and iterate till convergence.
The algorithm is simple to understand: The matrix multiplication
$A x_t$ amplifies $x_t$ in the leading eigenspace. The convergence of
the algorithm is also well known: As long as the initial vector
satisfies that $u_1^{\top} x_0$ is nonzero and there exists an
eigengap ($\lambda_1 > \lambda_2$), the subspace $\spn{x_t}$ converges
to the eigenspace $\spn{u_1}$ linearly as $t\to\infty$ with rate
proportional to $\lambda_2 / \lambda_1$.

Since only the convergence of subspace is of interest, the norm of the
vector $x_t$ plays no role. Hence the normalization step of power
iteration may be omitted 
\begin{equation*}
  v_{t+1} = A v_t. 
\end{equation*}
This is equivalent with the original power method. Of course, in
practical computations, the normalization is important to avoid issues
like arithmetic overflow.

\smallskip 

Motivated by the recent proposed algorithms in quantum chemistry
literature, in this work, we take the view point that we cannot afford
(or choose not to perform) the matrix-vector multiplication
$y_{t+1}=Ax_t$ exactly.  Among other applications, such a scenario
naturally arises when the dimension of the matrix $A$ is extremely
large, so that even storage of the vector $y_{t+1}$ (even in sparse
format) is too expensive. For example, this is a common situation for
FCI calculations in quantum chemistry since the dimension of the
matrix $A$ grows exponentially with respect to the number of electrons
in chemical system.

Thus, in power iteration, we would replace the matrix multiplication
step by a map
\begin{equation}  \label{eq:inexact_iteration}
 y_{t+1} = F_m(A, x_t). 
\end{equation}
Here, given the matrix $A$ and the current iterate $x_t$, the map
$F_m$, either deterministic or stochastic, outputs an approximation of
the product $y_{t+1} \approx A x_t$.  Different choices of $F_m$
corresponds to various recently proposed algorithms, as will be
discussed below.  We have used the index $m$ to indicate the
``complexity'' (computational cost) of $F_m$, the specific meaning
depends on the choice of the family of maps.  Replacing the
matrix-vector multiplication by \eqref{eq:inexact_iteration}, we get
Inexact Power Iteration Algorithm~\ref{alg:inexact_power} and its
unnormalized version \ref{alg:inexact_power_no}.
\renewcommand\theContinuedFloat{\alph{ContinuedFloat}}

\begin{algorithm}[H]\ContinuedFloat*
  \SetAlgoLined
  Initialization: Choose a normalized vector $x_0\in \mathbb{R}^N$, $\norm{x_0}_2 = 1$, $u_1^\top x_0 \neq 0$.\\
  \For{$t=0,1,2,\cdots$, while not converged}{
    $y_{t+1} = F_m(A,x_{t})$\; $x_{t+1} = y_{t+1}/ \norm{y_{t+1}}_2$\;
  }
  \caption{Inexact Power Iteration}\label{alg:inexact_power}
\end{algorithm}

\begin{algorithm}[H]\ContinuedFloat
  \SetAlgoLined
  Initialization: Choose a vector $v_0\in \mathbb{R}^{N}$, $u_1^\top v_0 \neq 0$.\\
  \For{$t=0,1,2,\cdots$, while not converged}{
    $v_{t+1} = F_m(A,v_{t})$\;
  }
  \caption{Inexact Power Iteration without Normalization} \label{alg:inexact_power_no}
\end{algorithm}

Notice that the two versions of inexact power iteration Algorithm
\ref{alg:inexact_power} and \ref{alg:inexact_power_no} are equivalent
if the function $F_m(A, \cdot)$ is homogeneous; we will make this as a
standing assumption in our analysis.

Various inexact matrix-vector multiplication has been proposed in the
literature for configuration interaction calculations, either
deterministic or stochastic, see e.g., earlier attempts in
\cite{Huron1973, Buenker1974, Harrison1991, Illas1991, Daudey1993,
  Greer1995, Greer1998, Ivanic2003, Abrams2005}, the FCIQMC approach
\cite{BoothThomAlavi:09, BoothGruneisKresseAlavi:13,
  ClelandBoothAlavi:10, Booth2011, Schwarz2017}, the semi-stochastic
approach \cite{PetruzieloHolmesChanglaniNightingaleUmrigar:12,
  Blunt2015, HolmesChanglaniUmrigar:16, Sharma2017}, other stochastic
approaches \cite{Ten-no2013, Giner2013, LimWeare:17}, and various
deterministic strategies for compressed or truncated representation of
the wave functions \cite{Rolik2008, Roth2009, Ma2011a,
  Evangelista2014, Knowles2015, Zhang2016,
  TubmanLeeTakeshitaHeadGordonWhaley:16, Liu2016, Schriber2016,
  Zimmerman2017, Zimmerman2017a}.  In this work, we will focus on two of
such strategies: the full configuration-interaction quantum Monte
Carlo (FCIQMC) \cite{BoothThomAlavi:09, ClelandBoothAlavi:10} and the
fast randomized iteration (FRI) \cite{LimWeare:17}. In some sense,
these methods represent two ends of the spectrum of the possibilities;
so that the analysis of those can be easily extended to other
methodologies. The FCIQMC uses interacting particles to represent the
vector $v_t$ and stochastic evolution of particles to represent the
action of the matrix $A$ on the vector $v_t$. The FRI on the other
hand is based on exact matrix-vector multiplication and stochastic
schemes to compress the resulting vectors into sparse ones with given
number of nonzero entries. These algorithms will be discussed and
analyzed in Section~\ref{sec:algorithms}, following the general
analysis framework we establish in Section~\ref{sec:conv}.

The rest of the paper is organized as the following. In
Section~\ref{sec:conv}, we provide a convergence analysis for a
generic class of inexact power iteration. In
Section~\ref{sec:algorithms}, we give more details of FCIQMC and FRI
and analyze them following the convergence analysis established in
Section~\ref{sec:conv}. In Section~\ref{sec:numerics}, we perform
numerical tests on 2D Hubbard model and some
chemical molecules to compare the various algorithms and to verify the analysis results. 

\section{General convergence analysis of inexact power iteration}\label{sec:conv}

As an advantage of taking a unified framework of various algorithms,
the convergence of those can be understood in a fairly generic way,
which also facilitates comparison of different proposed strategies.
In this section, we establish a general convergence theorem of the
inexact power iteration.

The convergence of the iteration to the desired eigenvector will be
measured in the angle of the vectors. Recall that the angle between
two vectors $v$ and $w$ is given by
\begin{equation}
  \theta(v,w) = \arccos
  \biggl( \frac{\abs{\langle v, w \rangle}}{\norm{v}_2\norm{w}_2} \biggr).
\end{equation}
From the definition, it is obvious that
$\theta(v,w) = \theta(a v, b w)$ for any vectors $v,w$ and real
numbers $a, b$. In view of this insensitivity of the constant multiple
of vectors in the error measure, if the inexact matrix-vector
multiplication $F_m(A,v_t)$ satisfies the homogeneity assumption
below, the two versions of inexact power iterations with or without
normalization Algorithm \ref{alg:inexact_power} and
\ref{alg:inexact_power_no} are equivalent.

\begin{assumption}[Homogeneity]
  \label{asmp:F_homo}
  \begin{equation}
    F_m(A,cv) = c F_m(A,v),
  \end{equation}
  for all vectors $v\in \mathbb{R}^N$ and real number $c\in \mathbb{R}$.
\end{assumption}

More precisely, if the initial vectors of the two algorithms are the
same up to a number $x_0 = c_0 v_0$, then there exist numbers $c_t$
such that $v_t = c_tx_t$ for all $t$. Therefore,
$\theta(u_1,v_t) = \theta(u_1,x_t)$.  In the following, when we
analyze the algorithm, we will always use $v_t$ for the unnormalized
iterate and $x_t$ for the normalized version
$x_t = v_t / \norm{v_t}_2$.

\smallskip 

To analyze the effect of the inexact matrix-vector multiplication, we
write $F_m(A, v_{t-1})$ as a sum of the exact matrix-vector product
with an error term
\begin{equation}\label{eq:inexactproduct}
  v_{t} = F_m(A,v_{t-1}) = Av_{t-1} + \xi_{t},
\end{equation}
where $\xi_t$ is the error of the inexact multiplication at step $t$,
and the dependence on $m$ is suppressed to keep the notation
simple. Note that $\xi_t$ can be either deterministic or stochastic
depending on the choice of $F_m$. For example, $\xi_t$ is
deterministic for the hard thresholding compression and stochastic for
both FCIQMC and FRI methods. While we will proceed viewing $v_t$ as a
stochastic process, the results apply to the deterministic case as
well.

Denote $\mathcal{F}_t = \sigma(v_1, v_2, \cdots, v_t)$ the
$\sigma$-algebra generated by $v_1, v_2, \cdots, v_t$. We assume that
error $\xi_t$ satisfies the following properties. Note that this
assumption holds for both FCIQMC and FRI algorithms, as we will prove
in Section~\ref{sec:algorithms}.
\begin{assumption}
  \label{asmp:error}
  The error $\xi_t$ in the inexact matrix-vector product
  \eqref{eq:inexactproduct} satisfies
  \begin{enumerate}[a)]
    \item Martingale difference sequence condition
      \begin{equation}
          \label{eq:martingale_difference}
          \EE (\xi_t \mid \mathcal{F}_{t-1}) = 0.
      \end{equation}
    \item Expectation $2$-norm bound
      \begin{equation}
          \label{eq:2_norm_bound}
          \EE (\norm{\xi_t}_2^2 \mid \mathcal{F}_{t-1}) \leq C_e\frac{\norm{A}_1^2 \norm{v_{t-1}}_1^2}{m},
      \end{equation}
      where $C_e$ is a constant that is scale invariant of $A$
      (\textit{i.e.}, it does not depend on the norm of $A$).
    \item Growth of expectation $1$-norm bound
      \begin{equation}
          \label{eq:1_norm_control}
          \EE (\norm{v_t}_1 \mid \mathcal{F}_{t-1}) \leq \norm{A}_1 \norm{v_{t-1}}_1.
      \end{equation}
  \end{enumerate}
\end{assumption}

A few remarks are in order to help appreciate the
Assumption~\ref{asmp:error}.  The martingale difference sequence
property is just assumed here for convenience, in fact the convergence
result extends to the biased case as we will see in
Corollary~\ref{coro:biased}.  The other two assumptions are more
essential.  Assumption~\ref{asmp:error}b indicates that the error of
the inexact matrix-vector product $F_m(A, v_{t-1})$ can be controlled
by the sparsity of the $v_{t-1}$, as the $1$-norm of $v_{t-1}$ is a
sparsity measure. This is a natural assumption considering that the
compression of a vector would be easier if the vector is more
sparse. The bound depends proportional to inverse of $m$, so that one
could control the error of the inexact matrix-vector multiplication at
the price of increasing the complexity. Note that $1/m$ dependence can
be understood as a standard Monte Carlo error scaling. More detailed
discussions can be found in Section~\ref{sec:algorithms} when specific
algorithms are analyzed. Assumption~\ref{asmp:error}c then assumes
that the sparsity is not destroyed by the error in the iteration;
since otherwise we will lose control of the accuracy of the inexact
matrix-vector multiplication.

\medskip 

We now state the convergence theorem for the inexact power iteration
Algorithms \ref{alg:inexact_power} and \ref{alg:inexact_power_no}. The
theorem provides a convergence guarantee with high probability given
that the complexity parameter is sufficiently large with the number of
iteration steps $T$ chosen properly. Note that the logarithmic
dependence of $T$ on the spectral gap $\lambda_1 / \lambda_2$ and the
error criteria $\delta$ and $\varepsilon$ is expected from the
standard power method. The dependence of $m_0$, the complexity
parameter, on the ratio of the $1$-norm and $2$-norm of $A$ is due to
the competition between the $1$- and $2$-norm growth of the iterate,
where the $1$-norm matters for the control of the error of the inexact
matrix-vector product.

\begin{theorem}
  \label{thm:main_conv}
  For the inexact power iteration
  Algorithm~\ref{alg:inexact_power_no}, under
  Assumption~\ref{asmp:error}, 
  for any precision $\varepsilon > 0$ and small probability
  $\delta \in (0,1)$, there exist time
  \begin{equation}\label{eq:choiceT}
    T =  \log (\lambda_1 / \lambda_2)^{-1} \log \left( \frac{2\sqrt{2}}{\sqrt{\delta}\varepsilon\cos\theta(u_1,v_0)} \right) 
  \end{equation}
  and measure of complexity
  \begin{equation}\label{eq:choicem0}
    m_0 = \frac{4C_e}{\delta \varepsilon^2 \bigl(\cos\theta(u_1,v_0)\bigr)^2} \frac{\norm{v_0}_1^2}{\norm{v_0}_2^2} \, T \biggl( \frac{\norm{A}_1}{\norm{A}_2}\biggr)^{2T},
  \end{equation}
  such that with probability at least $1 - 2\delta$, for any $m \geq m_0$, it holds
  \begin{equation}
    \tan\theta(u_1,v_T) \leq \varepsilon
  \end{equation}
  Moreover, if Assumption~\ref{asmp:F_homo} is satisfied, the same
  result holds for Algorithm~\ref{alg:inexact_power}.
\end{theorem}

Before proving the theorem, let us collect a few immediate
consequences of Assumption~\ref{asmp:error}. The proof is obvious and
will be omitted.
\begin{lemma}\label{lem:main_conv}
  If the error $\xi_t$ satisfies Assumption~\ref{asmp:error}, we
  have
  \begin{enumerate}[a)]
  \item The error is unbiased
    \begin{equation}
      \EE \xi_t = 0.
    \end{equation}
  \item The error at different step is uncorrelated, in particular
    \begin{equation}
      \EE\xi_t^{\top} A^r \xi_s = 0
    \end{equation}
    for any $t \neq s$ and for all non-negative integer $r$.
  \item The expected $2$-norm of the error can be controlled as
    \begin{equation}
      \label{eq:2_norm_control}
      \EE \norm{\xi_t}_2^2 \leq C_e \norm{A}_1^{2t} \frac{\norm{v_0}_1^2}{m}.
    \end{equation}
  \end{enumerate}
\end{lemma}

\smallskip 

\begin{proof}[Proof of Theorem~\ref{thm:main_conv}]
  From the iteration of Algorithm \ref{alg:inexact_power_no}, we obtain
  \begin{align*}
    v_t &= Av_{t-1} + \xi_t \\
        &= A^t v_0 + A^{t-1} \xi_1 + \cdots + A\xi_{t-1} + \xi_t.
  \end{align*}
  Since the error $\xi_t$ is unbiased, we have
  \begin{equation*}
    \EE v_t = A^t v_0.
  \end{equation*}
  We now control the variance of $v_t$. Since $\xi_t$ is uncorrelated,
  we have
  \begin{equation*}
    \EE v_t^{\top}v_t = v_0^{\top}A^{2t}v_0 + \EE \xi_1^{\top} A^{2t-2} \xi_1 + \cdots \EE \xi_{t-1}^{\top} A^2 \xi_{t-1} + \EE \xi_t^{\top}\xi_t.
  \end{equation*}
  Thus,
  \begin{equation*}
    \EE v_t^{\top}v_t - \EE v_t^{\top}\EE v_t = \EE \xi_1^{\top} A^{2t-2} \xi_1 + \cdots \EE \xi_{t-1}^{\top} A^2 \xi_{t-1} + \EE \xi_t^{\top}\xi_t.
  \end{equation*}
  Using Lemma~\ref{lem:main_conv}, we estimate
  \begin{align*}
    \abs{\EE v_t^{\top}v_t - \EE v_t^{\top}\EE v_t} &= \abs{\EE \xi_1^{\top} A^{2t-2} \xi_1 + \cdots + \EE \xi_{t-1}^{\top} A^2 \xi_{t-1} + \EE \xi_t^{\top}\xi_t} \cr
    &\leq \EE\abs{\xi_1^{\top} A^{2t-2} \xi_1} + \cdots + \EE\abs{ \xi_{t-1}^{\top} A^2 \xi_{t-1}} + \EE\abs{\xi_t^{\top}\xi_t} \cr
    &\leq \lambda_1^{2t-2} \EE\abs{\xi_1^{\top}\xi_1} + \cdots + \lambda_1^2 \EE\abs{ \xi_{t-1}^{\top}\xi_{t-1}} + \EE\abs{\xi_t^{\top}\xi_t} \cr
    &\leq C_e \frac{\norm{A}_1^2 \norm{v_0}_1^2}{m} \Bigl(\lambda_1^{2t-2} + \cdots + \lambda_1^2 \norm{A}_1^{2t-4} + \norm{A}_1^{2t-2}\Bigr) \cr
    & = C_e \frac{\norm{A}_1^2 \norm{v_0}_1^2}{m} 
      \frac{\norm{A}_1^{2t} - \norm{A}_2^{2t}}{\norm{A}_1^{2} - \norm{A}_2^{2}},
  \end{align*}
  where in the last step, we used the fact that
  $\lambda_1 = \norm{A}_2$ since $\lambda_1$ is the largest
  eigenvalue. Recall that for symmetric matrix,
  $\norm{A}_1 \geq \norm{A}_2$, hence we have 
  \begin{equation}\label{eq:varest1}
    \abs{\EE v_t^{\top}v_t - \EE v_t^{\top}\EE v_t} 
    \leq     C_e \frac{\norm{v_0}_1^2}{m} \; t \norm{A}_1^{2t}. 
  \end{equation}
  By analogous arguments for 
  \begin{equation*}
    \EE v_tv_t^{\top} - \EE v_t\EE v_t^{\top} = \EE A^{t-1} \xi_1 \xi_1^{\top} A^{t-1} + \cdots + \EE A \xi_{t-1} \xi_{t-1}^{\top} A + \EE \xi_t\xi_t^{\top},
  \end{equation*}
  we get 
  \begin{equation}\label{eq:varest2}
    \bigl\lVert \EE v_tv_t^{\top} - \EE v_t\EE v_t^{\top} \bigr\rVert_2 \leq C_e \frac{ \norm{v_0}_1^2}{m} \; t  \norm{A}_1^{2t}.
  \end{equation}

  Let us now estimate the angle between $v_t$ and $u_1$ -- the
  eigenvector associated with the largest eigenvalue.  
  By definition, 
  \begin{equation}
    \label{eq:tan}
    \bigl( \tan\theta(u_1,v_t) \bigr)^2 = \frac{\norm{v_t}_2^2 - (u_1^{\top} v_t)^2 }{(u_1^{\top} v_t)^2}.
  \end{equation}
  For the denominator, we know the expectation
  \begin{equation*}
    \EE u_1^{\top} v_t = u_1^{\top} A^t v_0 
    = \lambda_1^t (u_1^{\top} v_0 ),
  \end{equation*}
  and the variance 
  \begin{align*}
    \var(u_1^{\top} v_t) &= u_1^{\top} (\EE v_tv_t^{\top} - \EE v_t\EE v_t^{\top}) u_1 \cr
    &\leq \norm{\EE v_tv_t^{\top} - \EE v_t\EE v_t^{\top}}_2 
      \stackrel{\eqref{eq:varest2}}{\leq} C_e \frac{\norm{v_0}_1^2}{m}\; t  \norm{A}_1^{2t}.
    \end{align*}
    The Chebyshev inequality implies that
    \begin{equation*}
      \mathbb{P}\biggl(\Abs{u_1^{\top}v_t - \lambda_1^t u_1^{\top}v_0} \geq \sqrt{\frac{C_e t}{m\delta}}  \norm{v_0}_1 \norm{A}_1^t\biggr)  \leq \delta,
    \end{equation*}
    and hence, as
    $\abs{\lambda_1^t u_1^{\top}v_0} - \abs{u_1^{\top}v_t} \leq
    \abs{u_1^{\top}v_t - \lambda_1^t u_1^{\top}v_0}$,
    \begin{equation*}
      \mathbb{P} \biggl( \Abs{u_1^{\top}v_t} \leq \Abs{\lambda_1^t u_1^{\top}v_0} - \sqrt{\frac{C_e t}{m\delta}} \norm{v_0}_1 \norm{A}_1^{t} \biggr)  \leq \delta, 
    \end{equation*}
    or equivalently 
    \begin{equation*}
      \mathbb{P} \biggl( \Abs{u_1^{\top}v_t}^2  \leq \biggl(\Abs{\lambda_1^t u_1^{\top}v_0} - \sqrt{\frac{C_e t}{m\delta}} \norm{v_0}_1 \norm{A}_1^{t}\biggr)^2 \biggr)  \leq \delta, 
    \end{equation*}
    For the numerator of \eqref{eq:tan}, the expectation is
    \begin{align*}
      \EE \bigl( \norm{v_t}_2^2 - (u_1^{\top} v_t)^2\bigr)  &= \sum_{i=2}^N u_i^{\top} \EE v_tv_t^{\top} u_i \\
        &= \sum_{i=2}^N u_i^{\top} (\EE v_tv_t^{\top} - \EE v_t \EE v_t^{\top}) u_i + \sum_{i=2}^N (u_i^{\top} \EE v_t)^2 \\
        &\leq \tr(\EE v_tv_t^{\top} - \EE v_t \EE v_t^{\top}) + \lambda_2^{2t} \norm{v_0}_2^2 \\
        &= (\EE v_t^{\top}v_t - \EE v_t^{\top}\EE v_t) + \lambda_2^{2t} \norm{v_0}_2^2 \\
                                                            &\stackrel{\eqref{eq:varest1}}{\leq} C_e \frac{ \norm{v_0}_1^2}{m} \; t  \norm{A}_1^{2t} + \lambda_2^{2t} \norm{v_0}_2^2.
    \end{align*}
    By Markov inequality, for any $\delta \in (0,1)$
    \begin{equation*}
      \mathbb{P} \left( \norm{v_t}_2^2 - (u_1^{\top} v_t)^2 \geq \frac{1}{\delta} \Bigl( C_e \frac{ \norm{v_0}_1^2}{m} \; t  \norm{A}_1^{2t} + \lambda_2^{2t} \norm{v_0}_2^2 \Bigr)\right) \leq \delta.
    \end{equation*}
    Therefore,
    \begin{equation}
      \mathbb{P} \left(
      \bigl( \tan \theta(u_1, v_t) \bigr)^2 
      \leq \frac{1}{\delta} \frac{  C_e \frac{ \norm{v_0}_1^2}{m} \; t  \norm{A}_1^{2t} + \lambda_2^{2t} \norm{v_0}_2^2  }{ \Bigl( 
        \Abs{\lambda_1^t u_1^{\top}v_0} - \sqrt{\frac{C_e t}{m\delta}} \norm{v_0}_1 \norm{A}_1^{t}  \Bigr)^2}
      \right) \geq 1-2\delta.
    \end{equation}
    We can explicitly check then with the choices \eqref{eq:choiceT} and \eqref{eq:choicem0}, for $m \geq m_0$, we have 
    \begin{equation*}
      \mathbb{P} \bigl(\tan\theta(u_1, v_T) \leq \varepsilon\bigr) \geq 1-2\delta,
    \end{equation*}
    thus the claim of the theorem.
\end{proof}

As mentioned above, it is possible to drop the martingale difference
sequence condition in Assumption~\ref{asmp:error} and get a similar
result. The reason is that the second moment bound
\eqref{eq:2_norm_bound} can be used to control the bias of $\xi_t$.
We state this as the following theorem.
\begin{theorem}\label{coro:biased}
  For the inexact power iteration Algorithm~\ref{alg:inexact_power_no}, under the Assumptions~\ref{asmp:error}(b) and \ref{asmp:error}(c), for any precision $\varepsilon>0$ and small probability $\delta \in (0,1)$, there exist time
  \begin{equation}
    T =  \log (\lambda_1 / \lambda_2)^{-1} \log \left( \frac{4}{\sqrt{\delta}\varepsilon\cos\theta(u_1,v_0)} \right) 
  \end{equation}
  and measure of complexity
  \begin{equation}
    m_0 = \frac{8C_e}{\delta \varepsilon^2 \bigl(\cos\theta(u_1,v_0)\bigr)^2} \frac{\norm{v_0}_1^2}{\norm{v_0}_2^2} \, T^2 \biggl( \frac{\norm{A}_1}{\norm{A}_2}\biggr)^{2T},
  \end{equation}
  such that with probability at least $1 - 2\delta$, for any $m \geq m_0$, it holds
  \begin{equation}
    \tan\theta(u_1,v_T) \leq \varepsilon
  \end{equation}
  Moreover, if Assumption~\ref{asmp:F_homo} is satisfied, the same
  result holds for Algorithm~\ref{alg:inexact_power}.
\end{theorem}
\begin{proof} Note that 
    \begin{align*}
        u_1^\top v_t &= u_1^\top A^t v_0 + u_1^\top A^{t-1} \xi_1 + \cdots + u_1^\top A \xi_{t-1} + u_1^\top \xi_t \\
        &= \lambda_1^t u_1^\top v_0 + \lambda_1^{t-1} u_1^\top \xi_1 + \cdots + \lambda_1 u_1^\top \xi_{t-1} + u_1^\top \xi_t, 
    \end{align*}
    so we get
    \begin{align*}
      \EE(u_1^\top v_t - \lambda_1^t u_1^\top v_0)^2 &= \EE(\lambda_1^{t-1} u_1^\top \xi_1 + \cdots + \lambda_1 u_1^\top \xi_{t-1} + u_1^\top \xi_t)^2 \cr
      &= \sum_{i,j=1}^t \lambda_1^{2t-i-j} \EE u_1^\top \xi_i \xi_j^\top u_1 \cr
      &\leq \sum_{i,j=1}^t \lambda_1^{2t-i-j} \bigl(\EE\norm{\xi_i}_2^2 \EE\norm{\xi_j}_2^2\bigr)^{1/2} \cr
      &\leq C_e t^2 \norm{A}_1^{2t} \frac{\norm{v_0}_1^2}{m}.
    \end{align*}
    Moreover,
    \begin{align*}
      \EE \bigl( \norm{v_t}_2^2 - (u_1^{\top} v_t)^2\bigr)^2 &= \sum_{i=2}^N (u_i^\top A^t v_0)^2 + 2\sum_{i=2}^N\sum_{b=1}^t \EE u_i^\top A^t v_0 \xi_b^\top A^{t-b} u_i + \sum_{i=2}^N \sum_{a=1}^t\sum_{b=1}^t \EE u_i^\top A^{t-a}\xi_a\xi_b^\top A^{t-b} u_i \\
                                        &\leq \lambda_2^{2t} \norm{v_0}_2^2 + 2t \sqrt{C_e} \lambda_2^t \norm{v_0}_2 \norm{A}_1^t \frac{\norm{v_0}_1}{\sqrt{m}} + t^2 C_e \norm{A}_1^{2t} \frac{\norm{v_0}_1^2}{m} \\
                                        & \leq 2 \lambda_2^{2t} \norm{v_0}_2^2 + 2 t^2 C_e \norm{A}_1^{2t} \frac{\norm{v_0}_1^2}{m},
    \end{align*}
    where the Cauchy-Schwarz inequality is used in the last line.
    Thus we can again use the Markov inequality to bound both
    numerator and denominator on the right hand side of \eqref{eq:tan}
    to obtain the claimed result.
\end{proof}

\section{Algorithms}\label{sec:algorithms}

In this section, we will review two stochastic power iteration methods
recently proposed in the literature: full configuration-interaction
quantum Monte Carlo (FCIQMC) \cite{BoothThomAlavi:09} and fast
randomized iteration (FRI) \cite{LimWeare:17}.  They can be analyzed
in the same framework we established in the previous section. In
particular, we prove the convergence of the two algorithms using
Theorem \ref{thm:main_conv}. We focus on these two methods since in
some sense they represent two opposite ends of strategies inexact
matrix-vector multiplications. It is possible to combine the ideas and
get a zoo of different approaches, which possibly yield better
results; and our analysis can be extended to these as well. We will
also comment on two variants: \textit{i}FCIQMC and hard thresholding (HT),
closely related to the FCIQMC and FRI approaches.  

Without loss of generality, we will assume the matrix $A$ is close to
the identity matrix and thus the eigenvalues $\lambda_i$ are close to
$1$ (we can always scale and center the original matrix so that this
is true).

\subsection{Full configuration-interaction quantum Monte Carlo}

\subsubsection{Algorithm Description}
FCIQMC is an algorithm originated in quantum chemistry literature to
calculate the ground energy of a many-body electron system by a Monte
Carlo algorithm for the full configuration-interaction of the
many-body Hamiltonian \cite{BoothThomAlavi:09}.

Let the Hamiltonian be a real symmetric matrix
$H \in \mathbb{R}^{N\times N}$ under the Slater determinant basis. To
find the ground state (the smallest eigenvector) of $H$, we write
$A = I - \delta H$ for $\delta > 0$ sufficiently small and hence focus
on the largest eigenvalue of $A$; this can be viewed as a first order
truncation of the Taylor series of $e^{-\delta H}$. It is also
possible to construct other variants of $A$ from $H$, which we will
not go here.

The FCIQMC can be viewed as a stochastic inexact power iteration for
finding the largest eigenvector of $A$, which corresponds more
naturally to the unnormalized version of the inexact power iteration
(Algorithm~\ref{alg:inexact_power_no}).  In the algorithm, the vector
$v_t$ is not stored as a vector, but represented as a collection of
``signed particles'' $\{\alpha_t^{(i)}\}_{i=1}^{M_t}$, where $M_t$ is
the number of signed particles at iteration step $t$.  Each signed
particle $\alpha$ has two attributes: location
$l_\alpha \in \{1,2,\cdots,N\}$ and sign $s_\alpha \in \{1,-1\}$.
Denote $e_l \in \mathbb{R}^{N}$ the standard basis vector with value
$1$ at its $l$-th component and $0$ at every other component.  Then
each signed particle $\alpha$ represents a signed unit vector
$\alpha = s_\alpha e_{l_\alpha}$. The vector $v_t \in \mathbb{Z}^N$ is
given by the sum of all signed particles at time $t$:
\begin{equation}\label{eq:vectorsum}
  v_t = \sum_{i=1}^{M_t} \alpha_t^{(i)}.
\end{equation}
With some ambiguity of notation, we refer to both the set of particles
and the corresponding vector as $v_t$, connected by
\eqref{eq:vectorsum}. As we always assume that the particles with
opposite signs on the same location are annihilated (see the
annihilation step in the algorithm description below), the vector
$v_t$ uniquely determines the set of particles.

    
In FCIQMC, the inexact matrix-vector multiplication $F_m(A,v_t)$
consists of three steps of particle evolution: spawning, diagonal
death/cloning and annihilation. Write $A = A_d + A_o$ with $A_d$ the
diagonal part and $A_o$ the off-diagonal part. The spawning step
approximates $A_o v_t$; the diagonal death / cloning step approximates
$A_d v_t$, and the annihilation step sums up the results from the
previous two steps and approximates the summation
$A v_t = A_o v_t + A_d v_t$. The three steps will be described in more
details below.

\smallskip

\noindent \textbf{Spawning.} Each signed particle $\alpha$ (we
suppress the index of $\alpha_t^{(i)}$ to simplify notation) is
allowed to spawn a child particle to another location, corresponding
to a nonzero component of
$A_o\alpha = s_{\alpha} A_o(:,l_\alpha)$.\footnote{MATLAB notation
  $A(:, l)$ is used to denote the $l$-th column of $A$.}  The location
of spawning is chosen at random, with probability
$p_{\text{loc}}(l \mid l_{\alpha})$, which is chosen in the original
FCIQMC algorithm to be uniformly random over all nonzero components of
$A_o\alpha$ for some simple Hamiltonian $H$. In
  general, $p_{\text{loc}}(\cdot \mid l_{\alpha})$ can be more
  complicated; we refer readers to \cite{BoothThomAlavi:09} for more
  details.  In the following of the paper,
  $p_{\text{loc}}(\cdot \mid l_{\alpha})$ is assumed to be uniform
  distribution over all nonzero components of $A_o\alpha$, while our
  analysis can be extended to other choice of
  $p_{\text{loc}}(\cdot \mid l_{\alpha})$.

Once the location $l$ is chosen, $n$ (possibly $0$) children particles
are spawned with the same sign $s = \sgn(A_o(l, l_\alpha) s_\alpha)$ determined by the sign of vector entry
$(A_o\alpha)(l)$ and the particle $\alpha$. The location $l$ and number $n$ are stochastically
chosen such that the overall step gives an unbiased estimate of
$A_o\alpha$:
\begin{equation}
  \label{eq:fciqmc_spawn_unbiased}
  \EE (n s e_l \mid \alpha) = A_o\alpha.
\end{equation}
Please refer Algorithm~\ref{alg:fciqmc_spawn} for details.

\renewcommand\theContinuedFloat{\alph{ContinuedFloat}}
\begin{algorithm}[H]\ContinuedFloat*
  \SetAlgoLined
  \SetKwInOut{Input}{Input}
  \SetKwInOut{Output}{Output}

  \Input{Set of particles: $\{ \alpha^{(i)}\}_{i = 1, \ldots, M}$,
    matrix $A$}

  \Output{New set of particles after spawning:
    $\{\alpha^{(j),\, \text{sp}}\}_{j = 1, \ldots, M^{\text{sp}}}$} 

  $M^{\text{sp}} = 0$;

  \For{each particle $\alpha\in\{\alpha^{(i)}\}_{i=1, \ldots, M}$}{

    Select a spawning location $l$ with probability
    $p_{\text{loc}}(l\mid l_\alpha)$;
    
    Determine the expected number of children
    \begin{equation*}
      Q = \frac{\abs{A_o(l,l_\alpha)}}{p_{\text{loc}}(l\mid
        l_\alpha)};
    \end{equation*}
      
    Randomly choose the number of children
    \begin{equation*}
      n = 
      \begin{cases}
        \lfloor Q \rfloor  & \text{w.p.}\quad 1  - (Q - \lfloor Q \rfloor) \\
        \lfloor Q \rfloor + 1 & \text{w.p.}\quad Q - \lfloor Q
        \rfloor
      \end{cases}
    \end{equation*}
      
    Assign sign of each children as
    $s = \sgn(A_o(l, l_\alpha) s_\alpha) = \sgn\bigl((A_o \alpha)(l)\bigr)$;

    Increase the number of particles
    $M^{\text{sp}} = M^{\text{sp}} + n$;

    Add $n$ particles with
    location $l$ and sign $s$ into the spawning set
    $\{\alpha^{(j), \text{sp}}\}$.  
  }
      
  \caption{FCIQMC - Spawning} \label{alg:fciqmc_spawn}
\end{algorithm}

\smallskip 
\noindent\textbf{Diagonal cloning / death.} 
This step represents $A_d v_t$ as a collection of particles in an
analogous way to the spawning step. For every signed particle
$\alpha$, we would consider children particles on the location
$l_{\alpha}$ (i.e., the location of the new particles is chosen to be
$l_{\alpha}$) and obtain an unbiased representation
\begin{equation}
  \label{eq:diagonal_unbiased}
  \EE (n s e_{l_\alpha} \mid \alpha) = A_d\alpha.
\end{equation}
The details can be found in Algorithm~\ref{alg:fciqmc_diagonal_death},
the key steps are similar to Algorithm~\ref{alg:fciqmc_spawn}.

\begin{algorithm}[H]\ContinuedFloat
  \SetKwInOut{Input}{Input}
  \SetKwInOut{Output}{Output}

  \Input{Set of particles: $\{ \alpha^{(i)}\}_{i = 1, \ldots, M}$,
    matrix $A$}

  \Output{New set of particles after diagonal cloning / death: 
    $\{\alpha^{(j),\, \text{diag}}\}_{j = 1, \ldots, M^{\text{diag}}}$} 

  $M^{\text{diag}} = 0$;

  \For{each particle $\alpha\in\{\alpha^{(i)}\}_{i=1, \ldots, M}$}{

     Determine the expected number of children at $l_{\alpha}$
     \begin{equation*}
       Q = \abs{A(l_\alpha,l_\alpha)},
     \end{equation*}
        
     Randomly choose the number of children
     \begin{equation*}
       n = 
       \begin{cases} 
         \lfloor Q \rfloor & \text{w.p.}  \quad 1  - (Q - \lfloor Q \rfloor) \\
         \lfloor Q \rfloor + 1 & \text{w.p.} \quad Q - \lfloor Q
         \rfloor
       \end{cases}
     \end{equation*}
     
     Assign sign of each children as $s = \sgn\bigl((A_d \alpha)(l_{\alpha})\bigr)$;

     Increase the number of particles $M^{\text{diag}} = M^{\text{diag}} + n$; 

     Add $n$ particles with location $l_{\alpha}$ and sign $s$ into the set
     $\{\alpha^{(j), \text{diag}}\}$.
  }
  \caption{FCIQMC - Diagonal cloning / death} \label{alg:fciqmc_diagonal_death}
\end{algorithm}

\smallskip 
\noindent\textbf{Annihilation.}     
The annihilation step merges the children particles from the previous
two steps and remove all pairs of particles with the same location and
opposite signs. If we denote $v^{\text{sp}}$ and $v^{\text{diag}}$ the corresponding vector representation of the particles are the spawning and diagonal cloning / death steps, the annihilation steps create a collection of particles representing the new vector $v = v^{\text{sp}} + v^{\text{diag}}$. 
Applying the three steps above to the particles representing $v_t$, we obtain the new set of particles $v_{t+1}$ at time $t+1$. Since by construction
\begin{equation*}
  \EE ( v^{\text{sp}} \mid v_t)  = A_o v_t, \quad \text{and} \quad 
  \EE ( v^{\text{diag}} \mid v_t) = A_d v_t,
\end{equation*}
we have on expectation 
\begin{equation}
  \EE ( v_{t+1} \mid v_t) = A v_t.
\end{equation}
In terms of the notations used in the framework of inexact power
iteration, $v_{t+1}$ represented using particles can be viewed as the
approximate matrix-vector product $F_m(A, v_t)$:
\begin{equation}
  \label{eq:fciqmc_inexact_power_iteration}
  F_m(A, v_t) := v_{t+1} = \sum_{i=1}^{M_{t+1}} \alpha_{t+1}^{(i)} = Av_t + \xi_{t+1},
\end{equation}
where $\xi_{t+1}$ is introduced in the last equality to denote the
error from the approximate matrix-vector multiplication through the
stochastic particle representation. As we will show in the analysis
below, the accuracy of FCIQMC iteration is controlled by the number of
particles $M_t$; and thus it plays the role of the complexity
parameter $m$ in our general framework. We would drop the subscript
$m$ for $F_m$ in the sequel for FCIQMC, as the complexity parameter is
implicit.

\smallskip 

Now that we have defined the inexact matrix-vector multiplication
$F(A,v_t)$ in FCIQMC, we may apply this in the inexact power iteration
as Algorithm~\ref{alg:inexact_power_no}. However, this can be
problematic in practice. Recall that $A = I - \delta H$ is assumed
to be a perturbation of identity so its eigenvalue is around $1$.  If
the largest eigenvalue of $A$ is strictly larger than $1$, when the
signed particles become a good approximation to the leading
eigenvector, the number of particles $M_t$ will grow exponentially
with rate $\lambda_1$, which quickly increases the computational cost
and memory requirement. It is also possible (while the probability is
tiny) that the number of particles may decrease to $0$ due to the
randomness.

In practice, it is desirable to have controls on the number of
particles to make the algorithm more stable.  One such strategy is to
introduce a shift $s_t \in \mathbb{R}$ and use matrix
\begin{equation}\label{eq:dynshift}
  \widetilde{A} = A + \delta s_t I = I - \delta ( H - s_t I)
\end{equation}
instead of $A$ at the $t$-th
step. Notice that $s_t$ only shifts the eigenvalues while not changing
the eigenspace. The shift $s_t$ is adjusted dynamically to control the
number of particles. With such shifts, the full FCIQMC algorithm is
presented in Algorithm~\ref{alg:fciqmc}.

\addtocounter{algocf}{-1}
\begin{algorithm}[H]
  \SetAlgoLined
  Initialization: $t=0$ and set initial particles $v_0$. \\

  \While{$M_t \leq M^{\text{target}}$, the target population}{
    \tcp{Phase 1: FCIQMC with fixed shift $s_0$} Spawning step: Use
    algorithm~\ref{alg:fciqmc_spawn} with $v_t$ and $A + \delta s_0 I$
    to get particle set $v^{\text{sp}}$;

    Diagonal death / cloning step: Use
    algorithm~\ref{alg:fciqmc_diagonal_death} with $v_t$ and
    $A + \delta s_0 I$ to get particle set $v^{\text{diag}}$;

    Annihilation step to get the particle set of the next time step
    $v_{t+1} = v^{\text{sp}} + v^{\text{diag}}$;

    Update $M_t$ and set $t=t+1$; }

  Set $s_t = s_0$; \quad \tcp{Initialize the dynamic shift} 

   \While{$t<t_{\max}$}{
     \tcp{Phase 2: FCIQMC with dynamic shift $s_t$}

     Spawning step: Use algorithm~\ref{alg:fciqmc_spawn} with $v_t$ and $A + \delta s_t I$ to get particle set $v^{\text{sp}}$; 
    
     Diagonal death / cloning step: Use
     algorithm~\ref{alg:fciqmc_diagonal_death} with $v_t$ and
     $A + \delta s_t I$ to
     get particle set $v^{\text{diag}}$;
    
     Annihilation step: Merge the two set of particles 
     $v_{t+1} = v^{\text{sp}} + v^{\text{diag}}$;

    Update the shift $s_t$ as 
    \begin{equation}\label{eq:adjust_shift}
       s_t = 
       \begin{cases}
         \displaystyle
         s_{t-q} - \frac{\eta}{q} \bigl( \ln M_t - \ln M_{t-q}\bigr), & \text{if} \; t = 0\ (\mathrm{mod}\ q), \\
         s_{t-1}, & \text{otherwise};
       \end{cases}
     \end{equation}
    
    Update $M_t$ and set $t=t+1$; 
}
\caption{FCIQMC} \label{alg:fciqmc}
\end{algorithm}

The Algorithm~\ref{alg:fciqmc} contains two phases for different
strategies of choosing the shifts and thus controlling the particle
population. In Phase 1, the shift is fixed to be $s_0$, which is
chosen such that $\abs{A(i,i) - s_0} \geq 1$ for all $i$ so that the
particle number is most likely to grow exponentially till the target
population $M^{\text{target}}$. In the second phase, the shift is
dynamically adjusted, so to control the growth of the population by a
negative feedback loop.  The target number of population
$M^{\text{target}}$ is chosen to be sufficiently large that the
variance is small enough to ensure convergence. It plays the role as
the `complexity' $m$ in Theorem~\ref{thm:main_conv}. $\eta$ and $q$ are two parameters to control the fluctuation of number of particles. For the details
of the parameter choices, we refer the readers to the original paper
on FCIQMC \cite{BoothThomAlavi:09} for details.

\smallskip 
\noindent \textbf{Energy Estimator}.  Several estimators can be used
to estimate the smallest eigenvalue of $H$ based on the FCIQMC
Algorithm~\ref{alg:fciqmc}, which is just a linear transformation of
the largest eigenvalue $\lambda_1$ of $A$.  One estimator is simply
the shift $s_t$. When the algorithm converges, $v_t$ is approximately
proportional to the eigenvector $u_1$. Since $s_t$ is adjusted to
control the number of particles steady, the largest eigenvalue of
$A + \delta s_t I$ is approximately $1$, hence connecting $s_t$ with
the desired eigenvalue estimate, cf.~\eqref{eq:dynshift}.  The other
estimator we will consider is the projected energy estimator
\begin{equation*}
  E_t = \frac{v_*^{\top} Hv_t}{v_*^{\top} v_t}.
\end{equation*}
Here $v_*$ is some fixed vector, for example the Hartree-Fock state of
the system. It is clear that when $v_t$ becomes a good approximation
of the eigenvector $u_1$, $E_t$ gives a good estimate of the leading
eigenvalue. In the numerical examples, we will focus on the projected
energy estimator, since it can be applied to all algorithms we
consider in this work (while shift estimator is unique for FCIQMC, in practice, it gives similar results compared to the projected energy estimator).

\subsubsection{Convergence Analysis}

Since FCIQMC can be viewed as an inexact power iteration as in
\eqref{eq:fciqmc_inexact_power_iteration}, we apply
Theorem~\ref{thm:main_conv} to analyze the convergence of FCIQMC. For
simplicity, we will focus on the case that the shift is constantly
$0$, $s_t = 0$, since the shift does not affect the eigenvector which
is the main focus of Theorem~\ref{thm:main_conv}. The probability distribution
in the spawning step $p_{\textrm{loc}}(\cdot\mid l_\alpha)$ is assumed 
to be uniform distribution over all the 
nonzero entries of $A_o(:,l_\alpha)$. To avoid some
degenerate case, we will assume that each diagonal entry of $A$ is
non-zero and each column of $A$ has more than $2$ nonzero entries (so
there is at least one possible location for children particles in the
spawning step).

We now check the three conditions in Assumption~\ref{asmp:error}. The
unbiasedness is guaranteed by construction as discussed above for the
FCIQMC algorithm, we have
\begin{equation}
  \EE (v_{t+1} \mid \mathcal{F}_t) = A v_t, 
\end{equation}
or equivalently, the error $\xi_t$ is a martingale difference
sequence:
\begin{equation}
  \EE (\xi_{t+1} \mid \mathcal{F}_t) = 0.
\end{equation}
The expectation $2$-norm bound is established in the following
proposition.
\begin{prop}
  For the inexact matrix-vector multiplication
  \eqref{eq:fciqmc_inexact_power_iteration} in FCIQMC
  Algorithm~\ref{alg:fciqmc}, the error $\xi_t$ satisfies
  \begin{equation}
    \EE \bigl( \Norm{\xi_{t+1}}_2^2 \mid \mathcal{F}_t\bigr) \leq \left(\max_{1\leq k \leq n}(\|a_k\|_0  - 2)\|a_{o, k}\|_2^2 + \frac{1}{2} \right) \frac{\Norm{v_t}_1^2}{M_t},
  \end{equation}
  where $a_k = A(:,k)$ is the $k$-th column vector of $A$, and
  $a_{o, k}$ is the $k$-th column vector of $A_o$, thus $a_{o, k}$
  equals $a_k$ except for the $k$-th entry $a_{o, k}(k) = 0$.
\end{prop}
    
\begin{proof}
  Since each particle evolves independently,
  \begin{equation*}
    F(A,v_t) = F\Bigl(A,\sum_{i=1}^{M_t} \alpha_t^{(i)}\Bigr) = \sum_{i=1}^{M_t} F(A,\alpha_t^{(i)}).
  \end{equation*}
  Moreover $F(A,\alpha_t^{(i)})$ and $F(A,\alpha_t^{(j)})$ are
  independent for $i\neq j$ conditioned on $\mathcal{F}_t$.

  By construction, $F(A,\alpha_t^{(i)})$ is unbiased, \textit{i.e.},
  \begin{equation*}
    \EE \bigl(F(A,\alpha_t^{(i)}) - A\alpha_t^{(i)} \mid \mathcal{F}_t\bigr) = 0.
  \end{equation*}
  Therefore,
  \begin{align*}
    \EE ( \Norm{\xi_{t+1}}_2^2 \mid\mathcal{F}_t) &= \EE \biggl( \biggl\lVert \sum_{i=1}^{M_t} (F(A, \alpha_t^{(i)}) - A\alpha_t^{(i)}) \biggr\rVert_2^2 \mid\mathcal{F}_t \biggr) \cr
    &= \EE \biggl( \Bigl(\sum_{i=1}^{M_t} F(A,\alpha_t^{(i)}) - A\alpha_t^{(i)}\Bigr)^{\top} \Bigl(\sum_{j=1}^{M_t} F(A,\alpha_t^{(j)}) - A\alpha_t^{(j)}\Bigr) \mid\mathcal{F}_t\biggr) \cr
    &= \sum_{i=1}^{M_t} \EE \bigl( (F(A,\alpha_t^{(i)}) - A\alpha_t^{(i)})^{\top} (F(A,\alpha_t^{(i)}) - A\alpha_t^{(i)}) \mid\mathcal{F}_t\bigr) \cr
    &\qquad + 2 \sum_{1\leq i < j \leq M_t} \EE \bigl( (F(A,\alpha_t^{(i)}) - A\alpha_t^{(i)})^{\top} \mid\mathcal{F}_t\bigr)\; \EE \bigl( (F(A,\alpha_t^{(j)}) - A\alpha_t^{(j)}) \mid\mathcal{F}_t\bigr) \cr
    &= \sum_{i=1}^{M_t} \EE \Bigl(\bigl\lVert F(A,\alpha_t^{(i)}) - A\alpha_t^{(i)} \bigr\rVert_2^2 \mid\mathcal{F}_t\Bigr).
  \end{align*}
  Hence, it suffices to consider each particle individually. To
  simplify the notation, without loss of generality, let us consider a
  particle with $\alpha_t^{(i)} = e_k$ for some $k$. Since the
  spawning and diagonal cloning/death steps are independent and
  unbiased, we have the decomposition
  \begin{equation*}
    \EE ( \Norm{ F(A,e_k) - A e_k}_2^2 \mid\mathcal{F}_t) = \EE ( \Norm{ F(A_o,e_k) - A_o e_k}_2^2 \mid\mathcal{F}_t) + \EE ( \Norm{ F(A_d,e_k) - A_d e_k}_2^2 \mid \mathcal{F}_t).
  \end{equation*}
  
  For the spawning step, since $A_o e_k = a_{o, k}$, there are
  $\norm{a_{o, k}}_0$ locations to spawn. Remind that 
  $p_{\textrm{loc}}(\cdot\mid k)$ is assumed to be uniform distribution,
  so each location is chosen with probability $\frac{1}{\norm{a_{o,k}}_0}$.
  Following the Algorithm~\ref{alg:fciqmc_spawn}, we calculate that
  \begin{equation*}
    F(A_o, e_k) = 
    \begin{cases}
      \lfloor\Norm{a_{o,k}}_0 |a_k(j)|\rfloor \sgn(a_k(j)) e_j, &
      \text{w.p.} \quad \bigl(1 - \bigl( \Norm{a_{o,k}}_0
      |a_k(j)| - \lfloor \Norm{a_{o,k}}_0 |a_k(j)|\rfloor \bigr) \bigr) / \Norm{a_{o,k}}_0,  \\
      \bigl(\lfloor\Norm{a_{o,k}}_0 |a_k(j)|\rfloor + 1\bigr)
      \sgn(a_k(j)) e_j, & \text{w.p.} \quad \bigl( \Norm{a_{o,k}}_0
      |a_k(j)| - \lfloor \Norm{a_{o,k}}_0 |a_k(j)|\rfloor \bigr) /
      \Norm{a_{o,k}}_0,
    \end{cases}
  \end{equation*}
  for each $j$ such that $a_{o,k}(j) \neq 0$. Straightforward calculation yields
  \begin{align*}
    \EE \Norm{F(A_o, e_k) - A_o e_k}_2^2
    & = (\Norm{a_{o,k}}_0 - 1) \Norm{a_{o,k}}_2^2  \cr
    & \qquad + \frac{1}{\Norm{a_{o,k}}_0} \sum_{j,\, a_{o,k}(j) \neq 0} \bigl(\Norm{a_{o,k}}_0 a_k(j) - \lfloor 
\Norm{a_{o,k}}_0 a_k(j) \rfloor\bigr) \times \cr 
      & \hspace{10em} \times \Bigl(1 - \bigl(\Norm{a_{o,k}}_0 a_k(j) - \lfloor 
\Norm{a_{o,k}}_0 a_k(j) \rfloor\bigr)  \Bigr)  \cr
    &\leq (\Norm{a_{o,k}}_0 - 1) \Norm{a_{o,k}}_2^2 + \frac{1}{4}.
  \end{align*}
  
  For the diagonal cloning/death step, we have
  \begin{equation*}
    F(A_d, e_k) = 
    \begin{cases}
      \lfloor |a_k(k)| \rfloor \sgn(a_k(k)) e_k, & \text{w.p.}
      \quad 1 - \bigl( \abs{a_k(i)} - \lfloor \abs{a_k(i)} \rfloor \bigr); \\
      (\lfloor|a_k(k)| \rfloor + 1) \sgn(a_k(k)) e_k, & \text{w.p.}
      \quad \abs{a_k(i)} - \lfloor \abs{a_k(i)} \rfloor.
    \end{cases}
  \end{equation*}
  Therefore
  \begin{equation*}
    \EE ( \Norm{F(A_d,e_k) - A_d e_k }_2^2 \mid\mathcal{F}_t) = 
    \bigl( \abs{a_k(i)} - \lfloor \abs{a_k(i)} \rfloor \bigr) \Bigl(1 -
\bigl( \abs{a_k(i)} - \lfloor \abs{a_k(i)} \rfloor \bigr) \Bigr)  \leq \frac{1}{4}.
  \end{equation*}
  Summing up the contribution from the two steps, we arrive at
  \begin{equation*}
    \EE ( \norm{F(A, e_k) - A e_k}_2^2 \mid\mathcal{F}_t) \leq (\norm{a_k}_0 - 2)\Norm{a_{o,k}}_2^2 + \frac{1}{2},
  \end{equation*}
  where we used $\norm{a_{o, k}}_0 = \norm{a_k}_0 - 1$. Thus
  \begin{equation*}
    \EE ( \norm{F(A, v_t) - A v_t}_2^2 \mid\mathcal{F}_t) \leq M_t \Bigl(\max_{1\leq k\leq n} (\norm{a_k}_0 - 2)\Norm{a_{o,k}}_2^2 + \frac{1}{2}\Bigr).
  \end{equation*}
  Since $M_t = \norm{v_t}_1$, we can rewrite the above estimate as
  \begin{equation*}
    \EE (\norm{F(A, v_t) - A v_t}_2^2 \mid\mathcal{F}_t) \leq \frac{\norm{v_t}_1^2}{M_t} \Bigl(\max_{1\leq k\leq n} (\|a_k\|_0 - 2)\|a_{o,k}\|_2^2 + \frac{1}{2}\Bigr).
  \end{equation*}
\end{proof}

Here we emphasize the important role of the annihilation step in
FCIQMC reflected in the error analysis above. Only with the
annihilation step is $M_t = \norm{v_t}_1$ true, so that the growth of
error is controlled as in the last step of the proof. In general,
without annihilation, the error will be exponentially larger as
$\frac{M_t}{\norm{v_t}_1}$ grows exponentially even when $v_t$ is
close to the eigenvector $u_1$.  Suppose $v_t$ is approximately
$u_1$. Then $v_{t+1} \approx \lambda_1 v_t$. Therefore,
$\norm{v_{t+1}}_1 \approx \norm{A}_2 \norm{v_t}_1$. However for the
number of particles $M_t$ without annihilation,
$M_{t+1} \approx \norm{\abs{A}}_2 M_t$, where $\abs{A}$ is the
entry-wise absolute value of $A$. To see this, let us denote $v_t^+$
the vector represented by all the particles with positive sign and
$-v_t^-$ the vector represented by all the particles with negative
sign. Then $v_t = v_t^+ - v_t^-$. Denote
$\tilde{v}_t = v_t^+ + v_t^-$. Then $M_t = \norm{\tilde{v}_t}_1$
without annihilation. We can easily check that $\tilde{v}_t$ evolves
according to $\tilde{v}_{t+1} = \abs{A} \tilde{v}_t$. So finally,
$\tilde{v}_t$ will converge to the eigenvector of $\abs{A}$, and
$M_{t+1} \approx \norm{\abs{A}}_2 M_t$. Noticing that
$\norm{A}_2 \leq \norm{\abs{A}}_2 \leq \norm{A}_1$, we know
$\frac{M_t}{\norm{v_t}_1}$ grows exponentially at rate
$\frac{\norm{\abs{A}}_2}{\norm{A}_2}$ after convergence. Therefore if
the number of particles $M_t$ has an upper bound, which is always true
in practice due to computational resource constraint, $\norm{v_t}_1$ will decay
to zero exponentially, which means the algorithm will not converge to
the correct eigenvector.
Also comment that if the spawning distribution $p_{\text{loc}}(\cdot \mid l_{\alpha})$ is not exactly uniform distribution, then $\EE ( \Norm{ F(A_o,e_k) - A_o e_k}_2^2 \mid\mathcal{F}_t)$ will be bound by another constant depending on $A_o$. Therefore the bound of $\EE \bigl( \Norm{\xi_{t+1}}_2^2 \mid \mathcal{F}_t\bigr)$ in the Proposition will only differ by a constant multiplier.

Compared with Assumption~\ref{asmp:error}, we observe that the
particle number $M_t$ plays the role of the ``complexity''
parameter. The more particles we have, the smaller the error is.  We
have the following corollary assuming the particle number is bounded
from below by $m$
\begin{corollary}
  If the particle number satisfies $M_t \geq m$,
  \begin{equation}
    \EE (\norm{F(A, v_t) - A v_t}_2^2 \mid\mathcal{F}_t, M_t \geq m) \leq C_e \frac{\norm{A}_1^2 \norm{v_t}_1^2}{m},
  \end{equation}
  where
  $C_e = \dfrac{\max_{k} (\|a_k\|_0 - 2)\|a_{o,k}\|_2^2 +
    \frac{1}{2}}{\norm{A}_1^2}$ is a parameter scale-invariant of $A$.
\end{corollary}
 In summary, FCIQMC satisfies Assumption \ref{asmp:error}b, as long as
the particle number is not too small. Note that in practice the particle number can be controlled by the dynamic shift $s_t$ to ensure that it does not drop below the required lower bound.
    
\smallskip 

The Assumption~\ref{asmp:error}c, the growth of expectation $1$-norm
bound, can also be checked easily, since we have
\begin{align*}
  \EE (\norm{v_{t+1}}_1 \mid \mathcal{F}_t) &= \EE \Bigl(\biggl\lVert F(A, \sum_{i=1}^{M_t} \alpha_t^{(i)}) \biggr\rVert_1 \mid \mathcal{F}_t\Bigr) 
                                              = \EE \Bigl( \biggl\lVert\sum_{i=1}^{M_t} F(A,\alpha_t^{(i)}) \biggr\rVert_1 \mid \mathcal{F}_t\Bigr) \cr
  &\leq \sum_{i=1}^{M_t} \EE (\norm{F(A,\alpha_t^{(i)})}_1 \mid \mathcal{F}_t) 
    = \sum_{i=1}^{M_t}  \norm{A\alpha_t^{(i)}}_1 
    \leq \sum_{i=1}^{M_t} \norm{A}_1 \norm{\alpha_t^{(i)}}_1 
    = \norm{A}_1 \norm{v_t}_1.
\end{align*}
    
In conclusion, we have verified the assumptions of
Theorem~\ref{thm:main_conv}, and thus it can be applied for the
convergence and error analysis of FCIQMC.

\subsubsection{Remarks on \textit{i}FCIQMC} 

\textit{i}FCIQMC (initiator FCIQMC) \cite{ClelandBoothAlavi:10} is a modified
version of FCIQMC. It can be viewed as a bias-variance tradeoff
strategy to reduce the computational cost and error of the FCIQMC
approach, by restricting the spawning step.
  
The $n$ locations are divided into two sets: the initiators $L_i$ and
non-initiators $L_n$ with $L_i \cap L_n = \emptyset$,
$L_i \cup L_n = \{1,2,\cdots, N\}$. The rule of \textit{i}FCIQMC is that for
any particle $\alpha$ at a non-initiator location $l_\alpha \in L_n$,
it is only allowed to spawn children particles at locations already
occupied by some other particles. If $\alpha$ spawns particles to a
location unoccupied, then the children particles are discarded. An
exception rule is that if at least two particles at non-initiator
locations spawn children particles with the same sign at one
unoccupied location, then the children particles are kept. There are
no restrictions for spawning steps for particles in initiators.  In
the case that all the locations are initiators $L_n = \emptyset$,
\textit{i}FCIQMC reduces to FCIQMC.

The initiators $L_i$ are chosen at the beginning according to some
prior knowledge. The initiators are then updated at each step of
iteration. Suppose $n_{i,thre}\in\mathbb{N}$ is a fixed threshold. As
soon as the number of particles at a non-initiator location is greater
than the threshold $n_{i,thre}$, then the location becomes an
initiator.
Intuitively, initiators are more important locations for the
eigenvector since they are occupied by many particles. The
restrictions on the spawning ability of non-initiators reduce the
computational cost and the variance of the inexact matrix-vector
product while only introducing small bias since there are few
particles on non-initiators. Therefore, \textit{i}FCIQMC can be viewed as a variance control technique for FCIQMC.

\subsection{Fast Randomized Iteration}

In this section, we provide a numerical analysis based on our general
framework for the convergence of the fast randomized iteration (FRI),
recently proposed in the applied mathematics literature
\cite{LimWeare:17}, inspired by FCIQMC type algorithms. The basic idea
of the FRI method is to first apply the matrix $A$ on the vector of
current iterate, and then employ a stochastic compression algorithm to
reduce the resulting vector to a sparse representation.  The original
convergence analysis \cite{LimWeare:17} uses a norm motivated by
viewing the vectors as random measures. In comparison, as we have seen
in the proof of Theorem~\ref{thm:main_conv}, our viewpoint and
analysis is closer in spirit to numerical linear algebra, in
particular the standard convergence analysis of power method.

\subsubsection{Algorithm Description}
The fast randomized iteration (FRI) algorithm is based on a choice of
random compression function $\Phi_m : \mathbb{R}^N \to \mathbb{R}^N$,
which maps a full vector $v$ to a sparse vector $\Phi_m(v)$ with
approximately only $m$ nonzero components. The sparsity of $\Phi_m(v)$
reduces the storage cost of the vector and associated computational
cost. To combine FRI with the inexact power iteration, define
\begin{equation}
  \label{eq:fri_inexact}
  F_m(A, v_t) = \Phi_m(A v_t)
\end{equation}
in Algorithm \ref{alg:inexact_power} and
\ref{alg:inexact_power_no}. The error is
$\xi_{t+1} = \Phi_m(A v_t) - A v_t$.
    
Thus the FRI algorithm is completely characterized by the choice of
compression function $\Phi_m$, about which we assume the following
properties. These are adaptations of the
Assumptions~\ref{asmp:F_homo} and \ref{asmp:error} in the context of
a compression function.
\begin{assumption}
  \label{asmp:fri_error}
  For any vector $v \in \mathbb{R}^N$, the compression function
  $\Phi_m$ satisfies:
  \begin{enumerate}[a)]
  \item Homogeneity: For all $c \in \mathbb{R}$, 
    \begin{equation}
      \Phi_m(cv) = c\Phi_m(v);
    \end{equation}
  \item Unbiasedness
    \begin{equation}
      \EE (\Phi_m(v) \mid v) = v;
    \end{equation}
  \item Variance bound. For some constant $C_{\Phi}$ independent of $m$ and $v$, 
    \begin{equation}
      \EE (\norm{\Phi_m(v) - v}_2^2 \mid v) \leq C_{\Phi} \frac{\norm{v}_1^2}{m};
    \end{equation}
  \item Expectation $1$-norm bound
    \begin{equation}
      \EE (\norm{\Phi_m(v)}_1 \mid v) = \norm{v}_1.
    \end{equation}
  \end{enumerate}
\end{assumption}

The compression function $\Phi_m$ introduced in \cite{LimWeare:17} is as
follows. For a given vector $v\in \mathbb{R}^N$, first we sort the
entries as
$\abs{v(q_1)} \geq \abs{v(q_2)} \geq \cdots \geq \abs{v(q_N)}$, where
$q: [N] \to [N]$ is a permutation. The compression function consists
of two parts.  In the first part, large components of the vector are
preserved exactly. Define
\begin{equation*}
  \tau = \max_{ 1\leq i\leq N} \left\{ i : |v(q_i)| \geq \frac {\sum_{j=i}^N |v(q_j)|}{m+1-i} \right\},
\end{equation*}
with the convention $\max \{\emptyset\} = 0$, so $0\leq \tau \leq m$.
The compression function keeps the entries $v(q_i)$ for any
$1 \leq i \leq \tau$,
\begin{equation*}
  \bigl(\Phi_m(v)\bigr)(q_i) = v(q_i), \qquad \forall i \leq \tau.
\end{equation*}
Note that if $\norm{v}_0 \leq m$, all components are `large' and
$\Phi_m(v) = v$, the input vector is kept without compression.  The
remaining $n - \tau$ components are considered `small'. Under the
compression we only keep a few entries so the resulting vector
$\Phi_m(v)$ has about $m$ nonzero entries, as in Algorithm
\ref{alg:fri_compression}; the details are further discussed below.
\begin{table}[ht]
  \begin{algorithm}[H]
    \SetAlgoLined
    \SetKwInOut{Input}{Input}
    \SetKwInOut{Output}{Output}
    \Input{$v\in\mathbb{R}^N$, sparsity parameter $m$}
    \Output{$V = \Phi_m(v) \in\mathbb{R}^N$}
    \tcp{Part 1: Keep large components}
    $B = \{1,2,\cdots,N\}$;

    $s = \norm{v}_1$;

    $\displaystyle i' = \arg\max_{i\in B} \abs{v(i)}$;

    $\tau = 0$;

    \While{$\displaystyle \abs{v(i')} \geq \frac{s}{m-\tau}$}{
      \smallskip

      $V(i') = v(i')$; 

      $s = s - \abs{v(i')}$; 

      $\tau = \tau + 1$; 

      $B = B\backslash\{i'\}$; 

      $\displaystyle i' = \arg\max_{i\in B} \abs{v(i)}$;
    }
    \tcp{Part 2: Compress small components}

    \For{each $i\in B$}{
      Generate nonnegative random integer $\{N_i\}$, such that 
      \[\EE N_i = \frac{m-\tau}{s} \abs{v(i)}; \]

      $\displaystyle V(i) = \sgn(v(i)) N_i \frac{s}{m-\tau}$;
    }
    \caption{FRI - compression function $\Phi_m$} \label{alg:fri_compression}
  \end{algorithm}
\end{table}

In the second part of Algorithm~\ref{alg:fri_compression}, the set
$B = \{q_{\tau+1}, q_{\tau+2}, \cdots, q_N\}$ consists of the indices
of all `small' components to be compressed. Note that for the integer
random variable $N_i$, $i \in B$, only its expectation
$\EE N_i \in (0,1)$ is specified, so there is still freedom to choose
the probability distribution of $\{N_i\}_{i\in B}$.  Here we only
discuss independent Bernoulli (which is easy to
understand) and systematic sampling (which we use in the numerical
examples) approaches, while other choices are possible.  Let us focus
on the entries in $B$ and define $v' \in \mathbb{R}^n$ such that
$v'(i) = v(i) \bd{1}_{\{ i\in B \}}$. It follows that
$\norm{v'}_1 = \norm{v}_1 - \sum_{i=1}^{\tau } \abs{v(q_i)}$.
    
For the independent Bernoulli, $N_i$ is independent for
each $i \in B$ and follows the Bernoulli distribution as
\begin{equation*}
  N_i = 
  \begin{cases}
      0, & \text{w.p.} \quad  1 - \frac{|v(i)|}{\|v'\|_1 /(m - \tau)}\,; \\
      1, & \text{w.p.} \quad \frac{|v(i)|}{\|v'\|_1 /(m - \tau)}\,.
    \end{cases}
\end{equation*}
Note that the probability is well defined due to the choice of
$\tau$. The number of nonzero components of the compressed vector is
$\norm{\Phi_m(v)}_0 = \tau + \sum_{i\in B} N_i$. From the choice of
$N_i$, $\EE (\norm{\Phi_m(v)}_0 \mid v) = m$; so $m$ is the expected
sparsity of $\Phi_m(v)$.

Another choice is the systematic sampling \cite{LimWeare:17}: Take a
random variable $U$ uniformly distributed in $(0,1)$. Then for
$k=1,2,\cdots,m - \tau$, define
\begin{equation*}
  U_k = \frac{U+k-1}{m - \tau}.
\end{equation*}
Given $\{q'_1, q'_2, \cdots, q'_{N-\tau}\}$ any permutation of indices
in $B$, define
\begin{equation*}
  I_k = \max_{1\leq i\leq N-\tau} \biggl\{i: \sum_{j=1}^{i-1} |v(q'_i)| \leq U_k \|v'\|_1 < \sum_{j=1}^i |v(q'_i)|  \biggr\},
\end{equation*}
then $N_i$ is given by 
\begin{equation*}
  N_i =
  \begin{cases}
      1, & \text{if}\quad i = q'_{I_k} \ \text{for some}\ k, \\
      0, & \text{otherwise}.
    \end{cases}
\end{equation*}
Notice that by construction, the number of nonzero $N_i$s is exactly
$m - \tau$, therefore $\norm{\Phi_m(v)}_0 = m$. The $N_i$s generated
by systematic sampling is obviously correlated as only one random
number $U$ drives the generation. The two approaches will be analyzed
in the next section in the framework of inexact power iteration.
  
\subsubsection{Convergence Analysis}
    
We now apply Theorem~\ref{thm:main_conv} to analyze the convergence of
the FRI algorithm with either independent Bernoulli or
systematic sampling. Notice that we have the immediate result
\begin{lemma}
  Assumption \ref{asmp:fri_error} implies Assumptions
  \ref{asmp:F_homo} and \ref{asmp:error}.
\end{lemma}
Therefore it suffices to check Assumption \ref{asmp:fri_error} for the
compression function $\Phi_m$. Homogeneity is obvious. From the
construction of $\Phi_m$, the unbiasedness is guaranteed by the
expectation of $N_i$s, no matter which particular distribution is used
for $N_i$.
\begin{equation}
  \EE (\Phi_m(v) \mid v) = v.
\end{equation}
The variance bounds are proved in the following lemmas.
\begin{lemma}
  For FRI compression with either independent Bernoulli or systematic
  sampling,
  \begin{equation*}
    \EE (\norm{\Phi_m(v) - v}_2^2 \mid v) \leq \frac{\norm{v'}_1^2}{m - \tau} \leq \frac{\norm{v}_1^2}{m}.
  \end{equation*}
  Moreover, we have the almost sure bound for systematic sampling,
  \begin{equation*}
    \norm{\Phi_m(v) - v}_2^2 \leq \frac{2 \norm{v'}_1^2}{m - \tau} \leq \frac{2 \norm{v}_1^2}{m}, \quad \text{a.s.}
  \end{equation*}
\end{lemma}
It is not possible to obtain an almost sure bound as above for
independent Bernoulli, since for example it is possible that all the
Bernoulli variables are $1$, which gives large error
$\norm{\Phi_m(v)-v}_2^2 \geq
\frac{(N-2m+\tau)\norm{v'}_1^2}{(m-\tau)^2}$.
This Lemma thus implies the advantage of using the systematic sampling
strategy, which in practice gives smaller variance in general. We will only show numerical results using the systematic sampling strategy in the numerical examples later.  

\begin{proof}
  Since large components of $v$ are kept exactly by $\Phi_m(\cdot)$, we have
  \begin{equation*}
    \norm{\Phi_m(v) - v}_2^2 = \sum_{i=\tau+1}^N (\Phi_m(v)(q_i) - v(q_i))^2 
    = \sum_{i=\tau+1}^N \bigl( (\Phi_m(v)(q_i))^2 + v(q_i)^2 - 2\Phi_m(v)(q_i)v (q_i) \bigr). 
  \end{equation*}
  Take the expectation
  \begin{equation*}
  	\EE \bigl(\norm{\Phi_m(v)-v}_2^2 \mid v\bigr) = \EE \biggl( \sum_{i=\tau+1}^N (\Phi_m(v)(q_i))^2 \mid v \biggr) + \sum_{i=\tau+1}^N v(q_i)^2 - 2\sum_{i=\tau+1}^N v(q_i) \EE (\Phi_m(v)(q_i) \mid v).
  \end{equation*}
  Since both independent Bernoulli and systematic sampling are unbiased,
  \begin{equation*}
  	\EE \Phi_m(v)(q_i) = v(q_i).
  \end{equation*}
  Moreover, because there are $\sum_{i=\tau+1}^N N_i$ number of $\frac{\norm{v'}_1}{m - \tau}$ and $n-\tau-\sum_{i=\tau+1}^N N_i$ number of $0$ in $\{\abs{\Phi_m(v)(q_i)}\}_{i\in B}$, we have
  \begin{equation*}
    \EE \biggl( \sum_{i=\tau+1}^N (\Phi_m(v)(q_i))^2 \mid v \biggr) = \EE\biggl(\EE \biggl( \sum_{i=\tau+1}^N (\Phi_m(v)(q_i))^2 \mid \sum_{i=\tau+1}^N N_i \biggr) \mid v \biggr)
    = \frac{\norm{v'}_1^2}{(m - \tau)^2} \EE \biggl( \sum_{i=\tau+1}^N N_i \mid v\biggr).
  \end{equation*}
  For independent Bernoulli, $\EE \bigl( \sum_{i=\tau+1}^N N_i \mid v \bigr) = m-\tau$ and for systematic sampling, $\sum_{i=\tau+1}^N N_i = m-\tau$, so
  \begin{equation*}
    \EE \biggl( \sum_{i=\tau+1}^N (\Phi_m(v)(q_i))^2 \mid v\biggr) = \frac{\norm{v'}_1^2}{m - \tau}.
  \end{equation*}
  Finally,
  \begin{equation*}
    \EE \bigl(\norm{\Phi_m(v)-v}_2^2 \mid v\bigr) = \frac{\norm{v'}_1^2}{m - \tau} - \norm{v'}_2^2  \leq \frac{\norm{v'}_1^2}{m - \tau}.
  \end{equation*}
  
  We now show that
  $\frac{\norm{v'}_1}{m - \tau} \leq \frac{\norm{v}_1}{m}$, which
  follows from the fact that $\frac {\sum_{j=i}^N |v(q_j)|}{m+1-i}$ is
  nonincreasing in $i$ for $i \leq \tau$. Indeed, recall from the
  choice of $\tau$ that for $i \leq \tau$,
  $|v(q_i)| \geq \frac{ \sum_{j=i}^N |v(q_j)|}{m+1-i}$, which is
  equivalent to
  \begin{equation*}
    \frac{\sum_{j=i}^N |v(q_j)|}{m+1-i} \geq \frac{\sum_{j=i+1}^N |v(q_j)|}{m-i}.
  \end{equation*}
  Thus, combined with $\norm{v'}_1 \leq \norm{v}_1$, we arrive at 
  \begin{equation*}
    \EE (\norm{\Phi_m(v) - v}_2^2 \mid v) \leq \frac{\norm{v'}_1^2}{m - \tau}
    \leq \frac{\norm{v}_1^2}{m}.
  \end{equation*}
  
  \medskip
  
  Next we give the almost sure bound for systematic sampling. Note that if $N_i \neq 0$ for $i \in B$, since $\bigl(\Phi_m(v)\bigr)(q_i)$ and $v(q_i)$ have the same sign, we have 
  \begin{equation*}
    (\Phi_m(v)(q_i) - v(q_i))^2 \leq \Phi_m(q_i)^2 = \frac{\norm{v'}_1^2}{(m - \tau)^2}.
  \end{equation*}
  Since there are exactly $m-\tau$ nonzero $N_i$s, we can estimate
  \begin{align*}
    \norm{\Phi_m(v) - v}_2^2 &= \sum_{i=\tau+1}^N (\bigl(\Phi_m(v)\bigr)(q_i) - v(q_i))^2 \cr
    &\leq \sum_{i = \tau + 1}^N v(q_i)^2 \bd{1}_{N_i = 0} + \bigl(\Phi_m(v)\bigr)(q_i)^2 \bd{1}_{N_i \neq 0}\cr
    &\leq \norm{v'}_2^2 + (m - \tau)\frac{\norm{v'}_1^2}{(m - \tau)^2} \cr
    &\leq \frac{\norm{v'}_1^2}{m - \tau} + \frac{\norm{v'}_1^2}{m - \tau} =  
      \frac{2\norm{v'}_1^2}{m - \tau} \leq \frac{2\norm{v}_1^2}{m}.
  \end{align*}
\end{proof}

The expectation $1$-norm bound can be easily checked from the definition. 
\begin{lemma}
  For FRI with independent Bernoulli compression,
  \begin{equation*}
    \EE (\norm{\Phi_m(v)}_1 \mid v) = \norm{v}_1.
  \end{equation*}
  For FRI with systematic sampling compression,
  \begin{equation*}
    \norm{\Phi_m(v)}_1 = \norm{v}_1, \quad \text{a.s.}
  \end{equation*}
\end{lemma}

Therefore, the compression function $\Phi_m$ satisfies
Assumption~\ref{alg:fri_compression}, and thus the convergence follows
Theorem~\ref{thm:main_conv}.
  
\subsubsection{Deterministic compression by hard thresholding}

Another way to choose the compression function $\Phi_m$ is by simple
hard thresholding, which means $\Phi_m = \Phi_m^{\text{HT}}$ keeps the
$m$ largest entries (in absolute value) and drops the remaining
ones. Compared to the previously discussed approaches of compression,
the hard thresholding obviously has smaller variance since it is
deterministic, as a price to pay, it introduces bias to the inexact
matrix-vector multiplication. The bias-variance tradeoff between hard
thresholding and FRI type algorithm is similar to that between
\textit{i}FCIQMC and FCIQMC.

\section{Numerical Results}\label{sec:numerics}
In this section, we give some numerical tests of the FCIQMC and FRI
algorithms, and their variance \textit{i}FCIQMC and Hard Thresholding
to compare the performance. The numerical problem is to compute the
ground energy of a Hamiltonian $H$ for a quantum system. As discussed
before, we define $A = I - \delta H$ for $\delta$ small so the problem
is equivalent to find the largest eigenvalue of $A$.  We will test
these methods with two types of model systems: the 2D fermionic Hubbard
model and small chemical molecules under the full CI discretization.
The Hamiltonians for these have the same structure. Each electron
lives in a finite dimensional one-particle Hilbert space. The vectors
in the basis set of the one-particle Hilbert space are called
orbitals. The number of orbitals $N^{\text{orb}}$ is the dimension of
the one-particle space. We denote $N^{\text{elec}}$ the total number of
electrons in the system in total. Due to the Pauli exclusion
principle, there are at most two electrons with opposite spins in one
orbital. In our test examples, we choose the total spin
$S^{\text{tot}} = 0$.  Therefore the dimension of the space is
$\binom{N^{\text{orb}}}{N^{\text{elec}}/2}^2$, neglecting other
constraints like symmetry. The dimension grows exponentially as
$N^{\text{orb}}$ and $N^{\text{elec}}$ grows. Here we summarize the
system in our numerical tests in the Table~\ref{tb:system}:
  
\begin{table}[h]
  \centering
  \begin{tabular}{cccccc}
    \toprule
	System	 &	$N^{\text{orb}}$ &	$N^{\text{elec}}$ & 	dimension &			HF energy	& Ground energy \\ \toprule
    $4\times 4$ Hubbard &	16 & 	10 &	$1.2\times 10^6$ &	-17.7500	& -19.5809 \\
    \ce{Ne}, aug-cc-pVDZ 	& 	23 &	10 &	$1.4\times 10^8$ &	-128.4963	& -128.7114 \\
    \ce{H_2O}, cc-pVDZ 		& 	24 & 	10 & 	$4.5\times 10^8$ &  -76.0240 	& -76.2419\\
    \bottomrule
  \end{tabular}
  \caption{Test Systems} \label{tb:system}
\end{table}
  
The exact ground energy of the Hubbard model and \ce{Ne} are computed
using exact power iteration, and the ground energy of \ce{H_2O} is
from the paper~\cite{Olsen:98}.  We use \textsf{HANDE-QMC}\cite{Spencer2015}
(\url{http://www.hande.org.uk/}), an open source stochastic quantum
chemistry program written in Fortran, for FCIQMC and \textit{i}FCIQMC
calculation. FRI and HT subroutines are implemented in Fortran based
on \textsf{HANDE-QMC}. The Hamiltonian of the Hubbard model is included in the
\textsf{HANDE-QMC} package and the Hamiltonians of \ce{Ne} and
\ce{H_2O} are calculated using RHF (restricted Hartree Fock) by
\textsf{Psi4} (\url{http://www.psicode.org/}), an open source ab
initio electronic structure package.
The code to generate the entries of Hamiltonian
$H$ is the same for all four algorithms, so the comparison among
algorithms is fair in terms of computational time. The four algorithms
are tested on a computer with 6 Core Xeon CPU at 3.5GHz and 64GB RAM.

Note that our comparison is mostly for illustrative purpose and should
not be taken as benchmark tests for the various algorithms especially
for large scale calculations, which would depend on parallel
implementation, hardware infrastructure, etc. On the other hand, even
for small problems, the numerical results still  offer some
suggestions on further development of  inexact power iteration based
solvers for many-body quantum systems.

\subsection{Hubbard Model}
The Hubbard model is a standard model used in condensed matter
physics, which describes interacting particles on a lattice. In real
space, the Hubbard Hamiltonian is
\begin{equation}
  \label{eq:hubbard_real}
  H = -\sum_{\langle r,r'\rangle,\sigma} \hat{c}_{r,\sigma}^\dagger \hat{c}_{r',\sigma} + U\sum_r \hat{n}_{r\uparrow}\hat{n}_{r\downarrow},
\end{equation}
where we have scale the hopping parameter to be $1$ and so the on-site
repulsion parameter $U$ gives the ratio of interaction strength
relative to the kinetic energy. We choose an intermediate interaction
strength $U=4$ in our test.

In the $d$ dimensional Hubbard Hamiltonian \eqref{eq:hubbard_real}, $r$ is a $d$-dimensional vector
representing a site in the lattice, $\langle r,r'\rangle$ means $r$
and $r'$ are the nearest neighbor, and $\sigma$ takes values of
$\uparrow$ and $\downarrow$, which is the spin of the
electron. $\hat{c}_{r,\sigma}$ and $\hat{c}_{r,\sigma}^\dagger$ are
the annihilation and creation operator of electrons at site $r$ with
spin $\sigma$. They satisfy the commutation relations
\begin{equation*}
  \{\hat{c}_{r,\sigma}, \hat{c}_{r',\sigma'}^\dagger \} = \delta_{r,r'}\delta_{\sigma, \sigma'}, \qquad 
  \{\hat{c}_{r,\sigma}, \hat{c}_{r',\sigma'} \} = 0, \quad \text{and} \quad 
  \{\hat{c}_{r,\sigma}^\dagger, \hat{c}_{r',\sigma'}^\dagger \} = 0, 
\end{equation*}
where $\{A, B\} = AB + BA$ is the anti-commutator.
$\hat{n}_{r,\sigma}$ is the number operator and defined as
$\hat{n}_{r,\sigma} = \hat{c}_{r,\sigma}^\dagger \hat{c}_{r,\sigma}$.
We will consider Hubbard model on a finite 2D lattice with periodic
boundary condition.

When the interaction strength $U$ is small, it is better to work in
the momentum space instead of the real space, since the planewaves are
the eigenfunctions of the kinetic part of the Hamiltonian. The
annihilation operator in momentum space is
$\hat{c}_{k,\sigma} = \frac{1}{\sqrt{N^{\text{orb}}}} \sum_r
e^{ik\cdot r} \hat{c}_{r,\sigma}$,
where $k=(k_1,k_2)$ is the wave number and $N^{\text{orb}}$ is the total number
of orbitals or sites.  The Hubbard Hamiltonian in momentum space is
then
\begin{equation}
  \label{eq:hubbard_k}
  H = \sum_{k,\sigma} \varepsilon(k)n_{k,\sigma} + \frac{U}{N^{\text{orb}}} \sum_{k,p,q}c_{p-q,\uparrow}^\dagger c_{k+q,\downarrow}^\dagger c_{k,\downarrow}c_{p,\uparrow},
\end{equation}
where $\varepsilon(k) = -2\sum_{i=1}^2 \cos(k_i)$.

Written as a matrix, the Hubbard Hamiltonian in the momentum space is
just a real symmetric matrix with diagonal entries $\varepsilon(k)$
and off-diagonal either $0$ or $\pm\frac{U}{N^{\text{orb}}}$.  For
inexact power iteration, we take $A = I - \delta H$ with
$\delta = 0.01$.  In our numerical test, we will use the projected
energy estimator for the smallest eigenvalue of $H$; the projected
vector $v_*$ is chosen to be the Hartree-Fock state. The initial
iteration of all methods is also chosen as the Hartree-Fock state (a
vector whose only nonzero entry is at the Slater determinant corresponding to
the Hartree-Fock ground state of the system).
   
Figure~\ref{fig:convergence} plots the error of projected energy of
each iteration versus wall-clock time (first $1500$ seconds) for a
typical realization.  The error is defined as the difference between
the projected energy estimate and the exact ground energy. The
complexity parameters of the algorithms are shown in
Table~\ref{tb:4410hubbard}, which are chosen such that FRI and FCIQMC
use about the same amount of memory (e.g., the particle number in
FCIQMC is roughly equal to the non-zero entries of the matrix-vector
product in FRI or HT before compression), and also chosen so large
that all the algorithms converge. The time per iteration listed in
Table~\ref{tb:4410hubbard} is averaged over several realizations and
is used in the Figure~\ref{fig:convergence}.
\begin{figure}[h]
  \centering
  \includegraphics[width=0.9\textwidth]{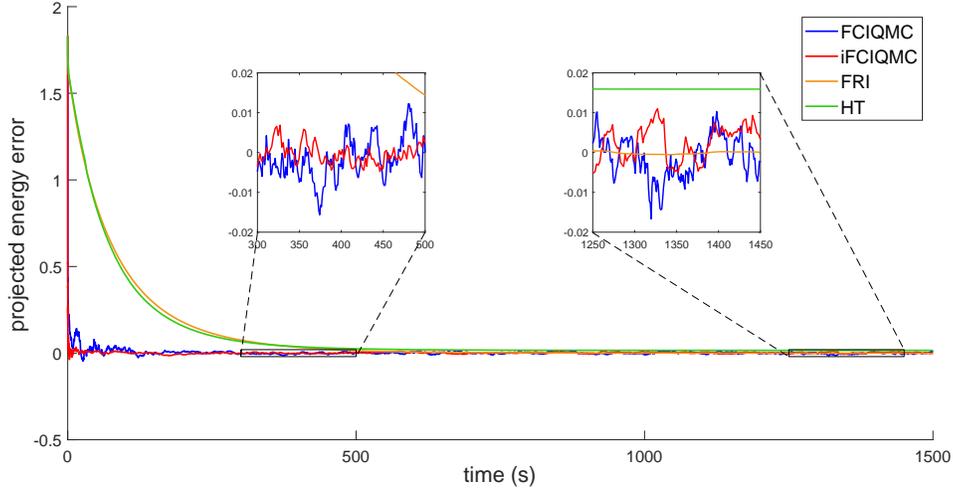}
  \caption{Convergence of the projected energy with respect to time
    for System 1, a $4\times 4$ Hubbard model with $10$ electrons, $5$
    spin up and $5$ spin down, and interaction strength $U=4$.}
  \label{fig:convergence}
\end{figure}

{\small
\begin{table}[h] 
  \centering
  \begin{tabular}{cccccccccc}
    \toprule
    & $m$				& $\norm{Av_t}_0$ 	& avg. error								& std. 					& MSE 					& $\tau_{\text{auto}}$		& compr. error			& time/iter.(s) \\
    \toprule
    FCIQMC	& $1.7\times 10^6$	& - 				& $4.4\times 10^{-4}$				& $3.0\times 10^{-4}$	& $2.6\times 10^{-7}$	& $14.1$ 			& $4.6\times 10^{-2}$	& $1.1$ \\
    \textit{i}FCIQMC & $1.7\times 10^6$	& -			& $3.2\times 10^{-4}$				& $2.2\times 10^{-4}$	& $1.8\times 10^{-7}$ 	& $13.8$			& $2.6\times 10^{-2}$	& $0.91$ \\
    FRI		& $3.0\times 10^4$	& $9.4\times 10^5$	& $1.2\times 10^{-4}$				& $6.1\times 10^{-5}$ 	 & $2.8\times 10^{-8}$	& $13.7$ 			& $1.6\times 10^{-1}$	& $3.6$ \\
    HT		& $3.0\times 10^4$	& $7.2\times 10^5$	& $1.6\times 10^{-2}$				& -					& $2.5\times 10^{-4}$		& - 			& $4.5\times 10^{-3}$	& $3.5$ \\
    \bottomrule
  \end{tabular}
  \caption{Parameters and numerical results for  $4\times 4$ Hubbard Model with $10$ electrons, $5$ spin up and $5$ spin down, and interaction strength $U=4$.} \label{tb:4410hubbard}
  \end{table}}

As shown in Figure~\ref{fig:convergence}, all four algorithms converge
to result close to the exact eigenvalue and the estimated value from
each iteration stays around the eigenvalue for a long time. FCIQMC and
\textit{i}FCIQMC take much less time to converge, thanks to their much
lower-cost inexact matrix-vector multiplication compared to FRI and
HT, but the variance is also larger. In terms of iteration number, the
convergence of the four algorithms is similar, which can be understood
from our analysis since it is the same eigenvalue gap of the
Hamiltonian that drives the convergence. As we mentioned already, per
iteration, the FCIQMC and \textit{i}FCIQMC is much cheaper in
comparison. The reason is that FRI and HT need to access all nonzero
elements of $A$ for each column associated with a non-zero entry in
the current iterate (for multiplying $A$ with the sparse vector),
while FCIQMC and \textit{i}FCIQMC just need to randomly pick some,
without accessing the others. The number of non-zero entries per row
is large and accessing elements of $A$ is quite expensive for FCI type
problems. More quantitatively, we see in Table~\ref{tb:4410hubbard}
that for a sparse vector of $3\times 10^4$ non-zero entries in FRI,
after multiplication by $A$ before compression, the number of non-zero
entries increases to roughly $10^6$. Thus for this problem, on
average, each column has about $40$ nonzero entries that FRI needs to
access, while FCIQMC algorithm only needs access of few entries after
the random choice.
   
After convergence, the projected energy of FCIQMC and \textit{i}FCIQMC
fluctuate around the exact ground state energy. Although
\textit{i}FCIQMC is biased, the bias is not large for the current
problem, while the variance is smaller than FCIQMC. So
\textit{i}FCIQMC is an effective strategy for bias-variance trade-off.
The projected energy of FRI also varies around the true energy, and
the variance is much smaller than FCIQMC or \textit{i}FCIQMC. HT is
deterministic and the projected energy shows no variance. However the
bias is also quite visible.
   
We can average the projected energy over the path to get a better
estimate. The variance of the estimator will decay to zero as we
include longer time period in the average.  Thus, due to unbiasedness,
the error of FCIQMC and FRI can be made smaller if we run for long
enough.
In Table~\ref{tb:4410hubbard}, we give more quantitative comparison of the results of the algorithms. 
  The quantities in the table are defined as below 
\begin{align*}  
  & \text{avg. error} 	&& \frac{1}{w} \sum_{i=i_0}^{i_0+w-1} \abs{E_i - E^{\text{true}}}	\\
  & \text{std.} 		&& \sqrt{\frac{1}{w-1} \sum_{i=i_0}^{i_0+w-1} \Bigl(E_i - \frac{1}{w} \sum_{j=i_0}^{i_0+w-1} E_i\Bigr)^2} \sqrt{\frac{1+2\tau_{\text{auto}}}{W}} \\
  & \text{MSE}	 	&& \text{avg. error}^2 + \text{std.}^2 \\
  & \tau_{\text{auto}}	&& \sum_{t=1}^{w-1} \frac{\frac{1}{w-1} \sum_{i=i_0}^{i_0+w-t-1} \Bigl(E_i - \frac{1}{w} \sum_{j=i_0}^{i_0+w-1} E_i\Bigr)\Bigl(E_{i+t} - \frac{1}{w} \sum_{j=i_0}^{i_0+w-1} E_i\Bigr)}{\frac{1}{w-1} \sum_{i=i_0}^{i_0+w-1} \Bigl(E_i - \frac{1}{w} \sum_{j=i_0}^{i_0+w-1} E_i\Bigr)^2}\\
  & \text{compr. error} && \frac{1}{w} \sum_{i=i_0}^{i_0+w-1} \frac{\norm{\xi_{t+1}}_2}{\norm{Av_t}_2}	\\
\end{align*}
Here $E^{\text{true}}$ is the true ground energy obtained by exact
power iteration, $i_0$ is a burn-in parameter and $w$ is the window
size of the average. For FCIQMC and \textit{i}FCIQMC, $w = 1600$ and
$i_0 = 2400$. For FRI and HT, $w = 400$ and $i_0 = 600$. The numerical
tests show that the quantities above are insensitive to the choice of
$w$ and $i_0$, as long as the algorithms indeed converge after $i_0$
steps and the window size $w$ is not too small. $\tau_{\text{auto}}$
is the integrated autocorrelation time and $W$ is the number of
iterations averaged.  The std.{} is short for the standard deviation
of the sample mean $\bar{E}^{(W)}$ defined as
$\bar{E}^{(W)} = \frac{1}{W} \sum_{i = i_0}^{i_0+W-1} E_i$. Since the
time cost per iteration of different algorithms is quite different, to
make a fair comparison, we take
$W = \frac{10000}{\text{time per iter.}}$ for each algorithm. It gives
the standard error of the sample mean if we run each algorithm for
$10000$ seconds after convergence. The mean square error (MSE) is
simply defined to incorporate the variance and bias together.
  
\begin{figure}[h]
  \begin{subfigure}{0.45\textwidth}
    \includegraphics[width=\textwidth]{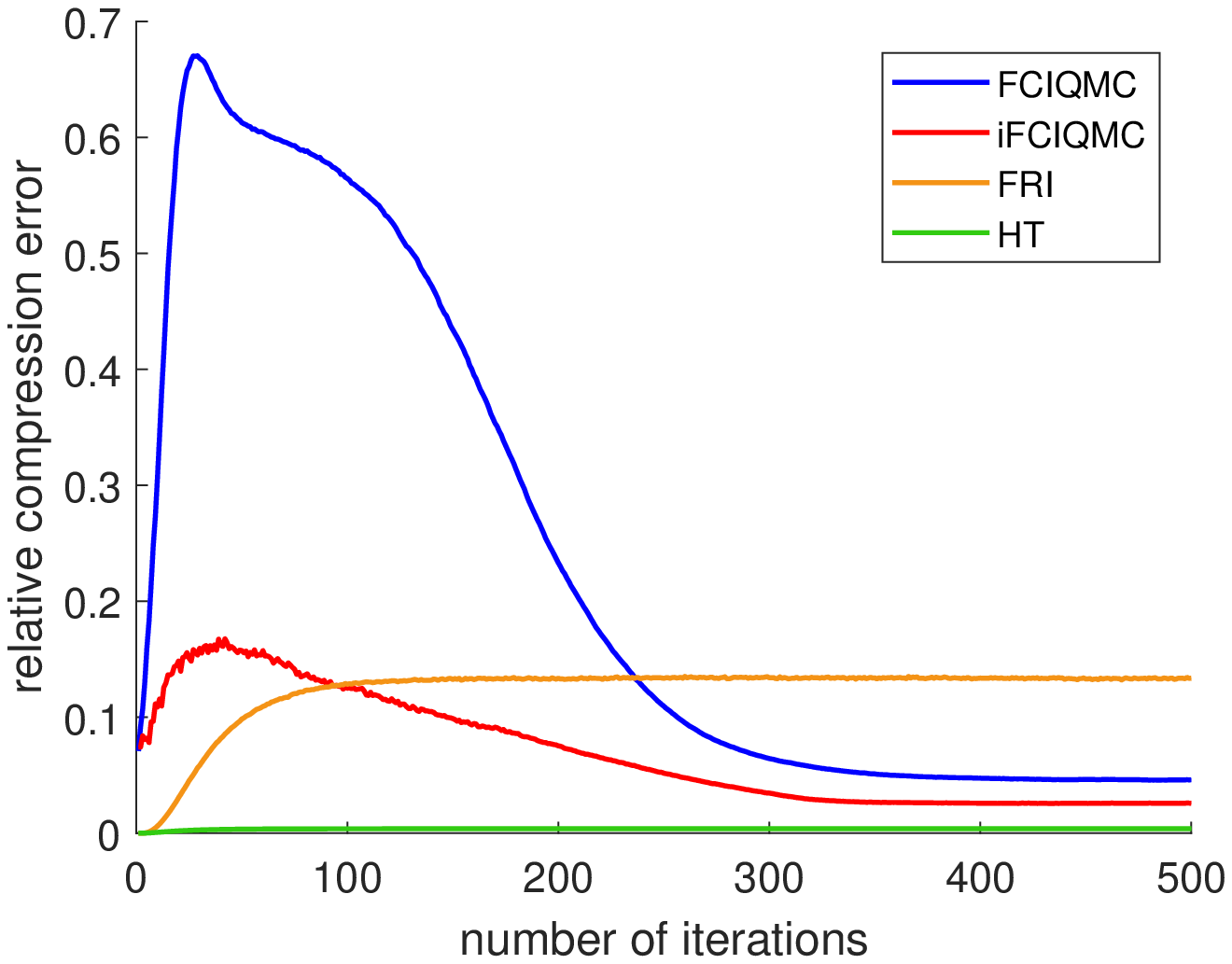}
  \end{subfigure}
  \begin{subfigure}{0.45\textwidth}
      \includegraphics[width=\textwidth]{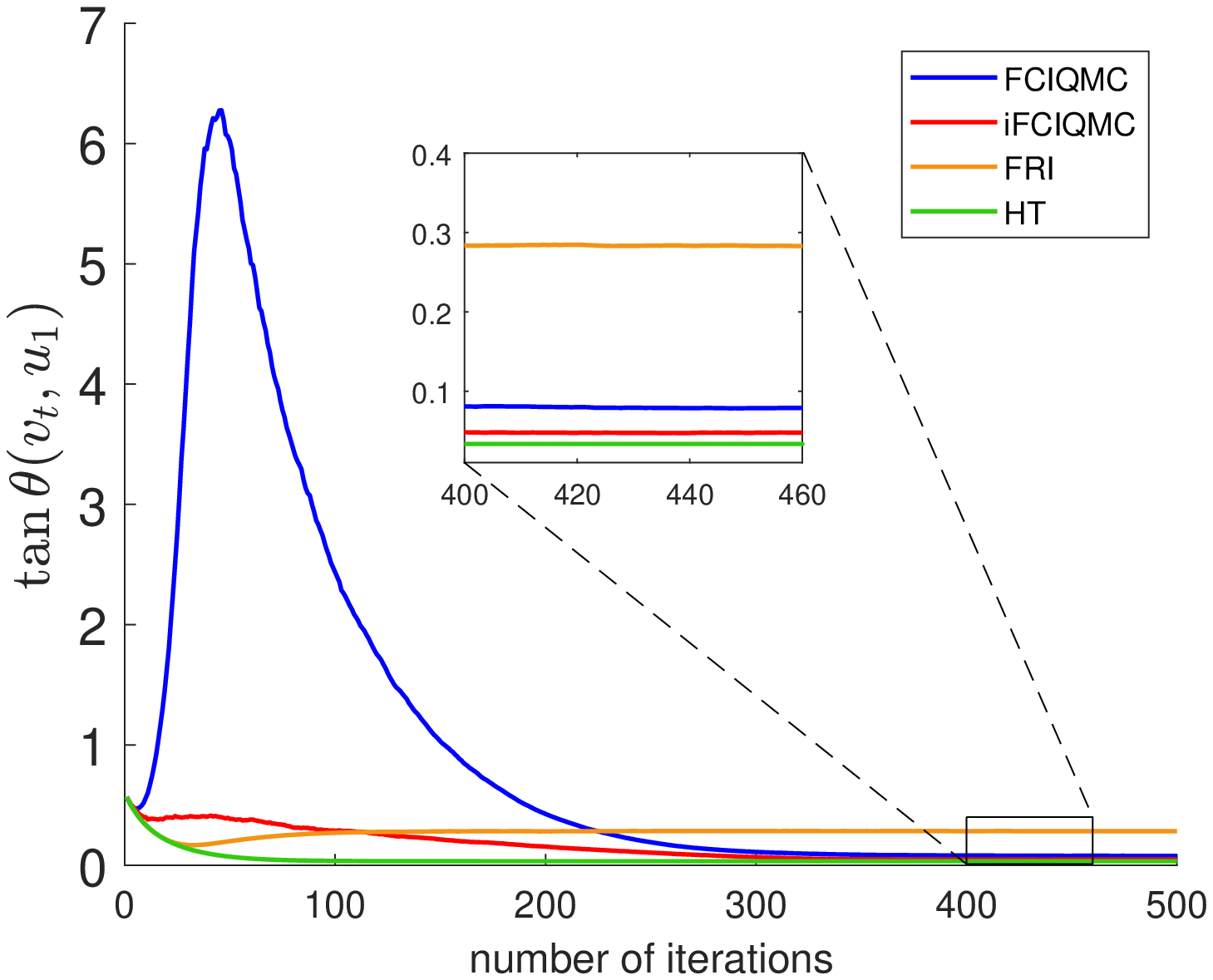}
  \end{subfigure}
  \caption{$4\times 4$ Hubbard model. (left) Relative compression
    error $\norm{\xi_{t+1}}_2 / \norm{Av_t}_2$ as a function of iteration steps; (right) Angle between the
    iterate and the exact ground state $\tan\theta(v_t,
    u_1)$ \label{fig:tot2}}
\end{figure}

To further obtain insights of the interplay between the error per step
of inexact power iteration and the convergence, we plot in
Figure~\ref{fig:tot2} the relative compression error and the tangent
of the angle between $v_t$ and the exact eigenvector $u_1$. We observe
that that FRI and HT reach convergence after about $100$ steps and
FCIQMC and \textit{i}FCIQMC converge after about $350$ steps; the more
steps of FCIQMC and \textit{i}FCIQMC are related with the first phase
of the algorithm where the particle number is exponentially
growing. This can be seen from Figure~\ref{fig:tot2}(left) as
the huge error growth of the initial stage of the iterations. Only
when the particle number reaches a certain level, the compression
error becomes small and the power iteration convergence kicks in.

After convergence, FRI has the largest compression error and HT has
the smallest. The compression error of \textit{i}FCIQMC is also
smaller than the one of FCIQMC. It is reasonable since HT and
\textit{i}FCIQMC reduce variance and thus compression error compared
with the fully stochastic FRI and FCIQMC. As shown in
Figure~\ref{fig:tot2}, in this example with the parameter choice,
FCIQMC has smaller compression error than FRI; and the larger the
compression error is, the further $v_t$ is away from the true
eigenvector $u_1$. This agrees with the theoretical results we obtain
in Theorem~\ref{thm:main_conv}, because $\tan\theta(v_t,u_1)$ is
controlled by the error $\xi_t$ at each step.

We remark that the $\tan\theta(v_t, u_1)$ error measure does not
directly translate to the error of the projected energy estimator
using say the Hartree-Fock state. In fact, we observe in
Figure~\ref{fig:convergence} and Table~\ref{tb:4410hubbard} that per
iteration, the projected energy estimated by FRI is smaller than
FCIQMC and \textit{i}FCIQMC. As an explanation, in our parameter
regime, the exact ground state has a large overlap with the
Hartree-Fock state, so in FRI, that component is kept unchanged in the
compression, while for FCIQMC and \textit{i}FCIQMC, the stochastic
error is more uniformly distributed over all the entries. This
behavior seems more problem dependent though, as we will see in the
chemical molecular examples that the MSE of FRI become comparable with
FCIQMC.

\subsection{Molecules}
We also tested the four algorithms for some molecule examples. The FCI
Hamiltonian is obtained by a Hartree-Fock calculations in a chosen
chemical basis (for single-particle Hilbert space), such as
cc-PVDZ. We choose \ce{Ne} and \ce{H_2O} at equilibrium geometry as
examples, which is described in Table~\ref{tb:system}. The time step
is taken as $\delta = 0.01$.
  
The convergence of projected energy error versus wall-clock time is
shown in Figure~\ref{fig:ne} and Figure~\ref{fig:h2o}
respectively. The parameter choice of the algorithms and more
quantitative comparison are shown in Table~\ref{tb:ne} and
Table~\ref{tb:h2o}. The four algorithms also work well for molecule
systems. The convergence behavior is similar to the Hubbard case.

The complexity parameter $m$ needed to achieve convergence depends on
the system. The ratio $m/N$ of \ce{Ne} is smaller than \ce{H_2O}. The
time cost of FRI and HT is much larger than FCIQMC and
\textit{i}FCIQMC, because they require the exact matrix-vector
multiplication $Av_t$, which is still expensive although $v_t$ is
sparse. Unlike the Hubbard case where FRI gives much smaller error,
the MSE of FRI is similar to FCIQMC and \textit{i}FCIQMC in these cases.
\begin{figure}[h]
  \centering
  \includegraphics[width=0.95\textwidth]{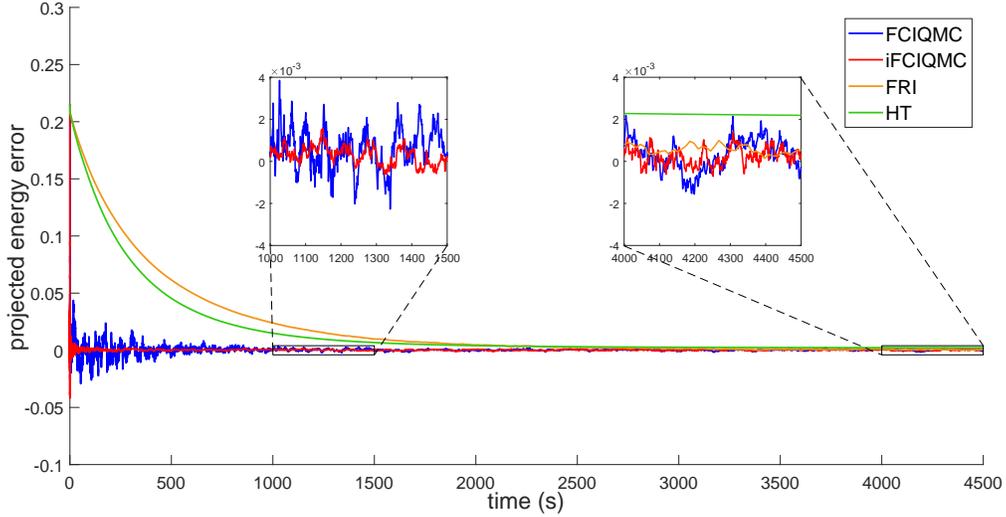}
  \caption{Convergence of the projected energy with respect to time
    for \ce{Ne} in aug-cc-pVDZ basis}
  \label{fig:ne}
\end{figure}

\begin{table}[h]
  \centering
  \begin{tabular}{cccccccc}
    \toprule
    & $m$				& $\norm{Av_t}_0$ 	& avg. error 		& std. 					& MSE 								& $\tau_{\text{auto}}$		&  time/iter.(s)	\\
    \toprule
    FCIQMC	& $1.8\times 10^6$	& - 				& $1.1\times 10^{-4}$	& $3.8\times 10^{-5}$ 	& $1.8\times 10^{-8}$	& $12.4$			& $1.5$					\\
    \textit{i}FCIQMC	& $1.8\times 10^6$ 	& - 	& $7.7\times 10^{-5}$	& $2.4\times 10^{-5}$	& $9.6\times 10^{-9}$	& $15.3$	& $1.1$					\\
    FRI		& $1.0\times 10^4$	& $7.9\times 10^6$	& $8.0\times 10^{-5}$	& $4.8\times 10^{-5}$	& $1.1\times 10^{-8}$	& $12.3$	& $11.6$				\\
    HT 		& $1.0\times 10^4$	& $3.3\times 10^6$	& $1.7\times 10^{-3}$	& -					& $2.8\times 10^{-6}$		& -	& $7.9$					\\
    \bottomrule
  \end{tabular}
  \caption{Comparison of algorithms for \ce{Ne} in aug-cc-pVDZ basis} \label{tb:ne}
\end{table}

\begin{figure}[h]
  \centering
  \includegraphics[width=0.95\textwidth]{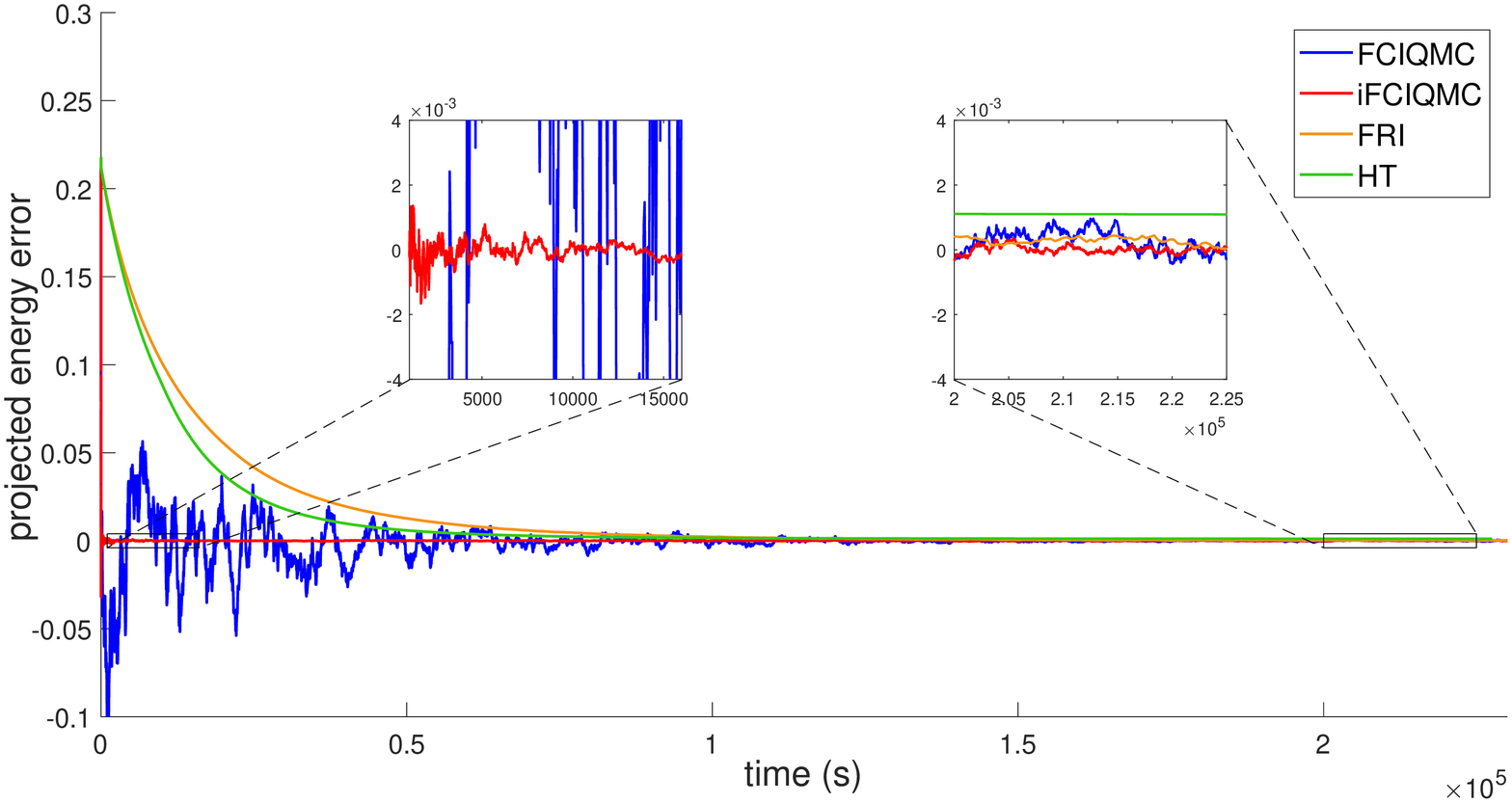}
  \caption{Convergence of the projected energy with respect to time
    for \ce{H_2O} in cc-pVDZ basis}
  \label{fig:h2o}
\end{figure}

\begin{table}[h]
  \centering
  \begin{tabular}{cccccccc}
    \toprule
    & $m$				& $\norm{Av_t}_0$ 	& avg. error 					& std.					& MSE 					& $\tau_{\text{auto}}$	&  time/iter.(s)	\\
    \toprule
    FCIQMC	& $6.0\times 10^7$	& - 				& $4.1\times 10^{-5}$	& $1.7\times 10^{-4}$	& $2.1\times 10^{-9}$ 	& $25.6$	& $54.1$	\\
    \textit{i}FCIQMC & $6.0\times 10^7$	& -			& $1.4\times 10^{-5}$ 	& $5.3\times 10^{-5}$	& $2.7\times 10^{-10}$ 	& $119$	& $36.6$	\\
    FRI 	& $1.2\times 10^5$	& $1.6\times 10^8$	& $2.2\times 10^{-5}$	& $1.2\times 10^{-4}$	& $8.9\times 10^{-10}$	 & $12.8$	& $379.3$	 \\
    HT 		& $1.2\times 10^5$	& $3.4\times 10^7$	& $1.1\times 10^{-3}$	& -					& $1.2\times 10^{-6}$	 & - 	& $227.0$	  \\
    \bottomrule
  \end{tabular}
  \caption{Comparison of algorithms for \ce{H_2O} in cc-pVDZ basis} \label{tb:h2o}
\end{table}

In summary, the numerical examples show that the FCIQMC, FRI and their
variants can achieve convergence using much less memory and
computational time compared to the standard power iteration. The
stochastic algorithms FCIQMC, iFCIQMC and FRI give better estimates
than the deterministic method HT in general.  The numerical test also
points out directions to further improve these inexact power
iterations, including variance and memory cost reduction of the
inexact matrix-vector multiplication and efficient parallel implementation to overcome the memory bottleneck. These will be leaved for future works. 

\bibliographystyle{amsxport}
\bibliography{qmc}

\end{document}